\documentclass[aop]{imsart}

\RequirePackage{amsthm, amsmath, amsfonts, amssymb, bbm, color, enumerate, graphicx, mathtools, tikz, hyperref, relsize}
\RequirePackage[numbers]{natbib}

\startlocaldefs
\newcommand{\NN}{\mathbb{N}}
\newcommand{\Om}{\Omega}
\newcommand{\RR}{{\mathbb{R}}}
\newcommand{\clf}{{\mathcal{F}}}
\newcommand{\cls}{{\mathcal{S}}}
\newcommand{\clp}{{\mathcal{P}}}
\newcommand{\Pmb}{{\mathbb{P}}}

\newcommand{\XX}{\mathbf{X}}
\newcommand{\Lb}{\mathbf{L}}
\newcommand{\tXX}{\widetilde{\mathbf{X}}}

\newcommand{\YY}{\mathbf{Y}}

\newcommand{\A}{\mathcal{A}}

\allowdisplaybreaks

\newtheorem{theorem}{Theorem}[section]
\newtheorem{corollary}[theorem]{Corollary}
\newtheorem{lemma}[theorem]{Lemma}
\newtheorem{proposition}[theorem]{Proposition}

\theoremstyle{definition}

\newtheorem{remark}[theorem]{Remark}

\let\plainqed\qedsymbol
\newcommand{\claimqed}{$\lrcorner$}

\usepackage{fancyhdr}

\usepackage{natbib}

\usepackage{amssymb}
\usepackage{amsmath}
\usepackage{amsthm}

\newcommand{{\LPC}}{\textbf{LPC}}

\newcommand{\R}{\mathbb{R}}
\newcommand{\Z}{\mathbf{Z}}
\newcommand{\tZ}{\widetilde{\mathbf{Z}}}
\newcommand{\U}{\mathbf{U}}
\newcommand{\V}{\mathbf{V}}
\newcommand{\Y}{\mathbf{Y}}
\newcommand{\z}{\mathbf{z}}
\newcommand{\vv}{\mathbf{v}}

\newcommand{\N}{\mathcal{N}}

\newcommand{\PP}{\mathbb{P}}

\newcommand{\equald}{\stackrel{\mathrm{d}}{=}}

\numberwithin{equation}{section}
\endlocaldefs

\begin{document}

\begin{frontmatter}
\title{Domains of attraction of invariant distributions of the infinite Atlas model}
\runtitle{Atlas model}

\begin{aug}
\author[A]{\fnms{Sayan} \snm{Banerjee}\ead[label=e1]{sayan@email.unc.edu}},
\author[A]{\fnms{Amarjit} \snm{Budhiraja}\ead[label=e2,mark]{budhiraj@email.unc.edu}}
\address[A]{Department of Statistics and Operations Research, 304 Hanes Hall, University of North Carolina, Chapel Hill, NC 27599, \printead{e1,e2}}

\end{aug}

\begin{abstract}
The infinite Atlas model describes a countable system of competing Brownian particles where the lowest particle gets a unit upward drift and the rest evolve as standard Brownian motions. The stochastic process of gaps between the particles in the infinite Atlas model does not have a unique stationary distribution and in fact for every $a \ge 0$, $\pi_a := \bigotimes_{i=1}^{\infty} \operatorname{Exp}(2 + ia)$ is a stationary measure for the gap process. We say that an initial distribution of gaps is in the weak domain of attraction of the stationary measure $\pi_a$ if the time averaged laws of the stochastic process of the gaps, when initialized using that distribution,   converge to $\pi_a$ weakly in the large time limit.  We provide general sufficient conditions on the initial gap distribution of the Atlas particles for it  to lie in the weak domain of attraction of $\pi_a$ for each $a\ge 0$. The cases $a=0$ and $a>0$ are qualitatively different as is seen from  the analysis and the sufficient conditions that we provide. Proofs  are based on the analysis of synchronous couplings, namely, couplings of the ranked particle systems started from different initial configurations, but driven using the same set of Brownian motions.
\end{abstract}

\begin{keyword}[class=MSC2020]
\kwd[Primary ]{60K35,82C22,60J60}
\kwd[; secondary ]{60J55}
\end{keyword}

\begin{keyword}
\kwd{Atlas model} 
\kwd{interacting diffusions}
\kwd{infinite dimensional diffusions}
\kwd{time averaged occupancy measures}
\kwd{synchronous couplings}
\kwd{local time}
\kwd{ergodicity}
\end{keyword}

\end{frontmatter}
\section{Introduction}\label{intro}
Rank-based diffusions comprise systems of Brownian particles where the drift and diffusivity of each particle depends on its relative rank in the system. Such diffusions are given by a system of stochastic differential equations (SDE) for $m \in \mathbb{N}_0 \cup\{\infty\}$, of the following form:  
\begin{equation}\label{rankbd}
dY_i(s) = \sum_{j=0}^m a_j \mathbf{1}(Y_i(s) = Y_{(j)}(s))ds + \sum_{j=0}^m b_j \mathbf{1}(Y_i(s) = Y_{(j)}(s))d W_i(s),
\end{equation}
for $s \ge 0,  i \in \{0,\dots,m\}$. Here, $a_j$ and $b_j$ respectively denote the drift and diffusion coefficients of the rank $j$ particle, $\{W_i(\cdot) : i \in \{0,\dots,m\}\}$ are independent standard Brownian motions,
$\YY(0) = \{Y_i(0)\}_{i=0}^m$ is a $\RR^m$-valued random variable independent of the above Brownian motions,
 and $\{Y_{(i)}(\cdot) : i \in \{0,\dots,m\}\}$ are the associated ranked (ordered) processes defined in an appropriate way.
Rank-based diffusions have gained significant attention in recent years as  models in Stochastic Portfolio Theory \cite{fernholz2002stochastic,fernholz2009stochastic}, scaling limits of interacting particle systems like the simple exclusion process \cite{karatzas2016systems}, and their connection with nonlinear diffusion processes \cite{dembo2016large,jourdain2008propagation}. \textcolor{black}{Rank-based diffusions also have  close connections with Aldous' ``Up the river'' stochastic control problem \cite{aldous41up}, which was recently solved in \cite{tang2018optimal}.}

The finite-dimensional setting $m < \infty$ is reasonably well-understood as in that case the system of gaps between the ranked particles can be described by a finite-dimensional reflected Brownian motion (RBM). One can then apply the general theory of RBM \cite{harrison1987brownian,harrison1987multidimensional,williams1995semimartingale} to study the process in a pathwise sense. In particular, it is known that the system in \eqref{rankbd}  has a unique weak solution \cite{bass1987uniqueness}. Moreover, the solution exists in a strong sense and is pathwise unique until the first time of a triple collision (three particles at the same location) \cite{ichiba2013strong}, and general criteria are known for the almost sure absence of triple collisions \cite{ichiba2010collisions,sarantsev2015triple,ichiba2017yet}. In particular, for the finite Atlas model, i.e. the case where  $a_j = \mathbf{1}(j=0), \ b_j \equiv 1$ for $j =0, 1, \cdots, m$, these criteria give unique strong solutions for \eqref{rankbd}.
Further, under very general stability and nondegeneracy conditions on the drift and diffusion coefficients $\{a_j\}, \{b_j\}$, the RBM describing gaps between the ranked particles has a unique stationary measure \cite{harrison1987brownian}. Geometric ergodicity results can be found in \cite{budlee} and rates of convergence to stationarity, depending explicitly on $\{a_j\}$, $\{b_j\}$ have recently been obtained \cite{BanBudh,banerjee2020dimension}.

Much less is known for the case $m=\infty$. Owing to the non-standard structure of the drift and diffusion coefficients, the infinite-dimensional rank-based diffusions are technically much more challenging to analyze and even some basic questions like uniqueness in law of the reflected process describing gaps between the ranked particles remain open in full generality. Moreover, unlike its finite-dimensional analogue, the gap process for infinite rank-based diffusions can have multiple stationary measures (like the model considered below) and the domains of attraction of these stationary measures are far from being \textcolor{black}{well-understood}. \textcolor{black}{See \cite{AS,DJO, dembo2017equilibrium, cabezas2019brownian} and 
references therein} for comprehensive surveys of the known results in the $m=\infty$ case. 

In this article, we study the long time behavior of the infinite Atlas model which, like its finite-dimensional analogue, is a rank-based diffusion where, at any point of time, the lowest particle has a unit upward drift and the remaining particles perform independent standard Brownian motions. This is a special case of the $\mathbb{R}^{\infty}$-valued process $\{Y_i(\cdot) : i \in \mathbb{N}_0\}$ defined in \eqref{rankbd} with $m = \infty$ and $a_j = \mathbf{1}(j=0), \ b_j \equiv 1$ for $j \in \mathbb{N}_0$.
Hence, $\{Y_i(\cdot) : i \in \mathbb{N}_0\}$ satisfies the following system of SDE:
\begin{equation}\label{Atlasdef}
dY_i(s) = \mathbf{1}(Y_i(s) = Y_{(0)}(s))ds + d W_i(s), \ s \ge 0,  i \in \mathbb{N}_0.
\end{equation}
It was shown in \cite[Theorem 3.2]{AS} that, if the starting points of the infinite Atlas model almost surely satisfy
\begin{equation}\label{S}
\sum_{i=0}^{\infty} e^{-\alpha Y_i(0)^2} < \infty, \mbox{ for all } \alpha>0,
\end{equation}
then the system \eqref{Atlasdef} has a weak solution that is unique in law. If \eqref{S} is satisfied, it can be shown that, almost surely, the system is locally finite in the sense that for any $T >0$, $u>0$, there exist only finitely many $i \in \mathbb{N}_0$ such that $\min_{s \in [0,T]} Y_i(s) \le u$ \cite[Lemma 3.4]{AS}. Thus, almost surely, the ranking $(Y_0(t),Y_1(t),\dots) \mapsto (Y_{(0)}(t),Y_{(1)}(t),\dots)$ is well defined for all $t \ge 0$ by prescribing the following convention for ties: if there exist $t \ge 0$ and $i<j$ such that $Y_i(t) =  Y_j(t)$, we assign a lower rank to $Y_i$ and higher rank to $Y_j$ at time $t$. We will assume throughout this article that $0 \le Y_{0}(0) \le Y_{1}(0) \le Y_{2}(0) \le \dots$.


By \cite[Lemma 3.5]{AS}, the processes defined by
\begin{equation}\label{eq:bistar}
B^*_i(t) := \sum_{j=0}^{\infty}\int_0^t\mathbf{1}(Y_j(s) = Y_{(i)}(s)) dW_j(s), \ i \in \mathbb{N}_0, t \ge 0,
\end{equation}
are independent standard Brownian motions which can be used to write down the following stochastic differential equation for $\{Y_{(i)}(\cdot) : i \in \mathbb{N}_0\}$:
\begin{equation}\label{order_SDE}
dY_{(i)}(t) = \mathbf{1}(i=0) dt + dB^*_i(t) - \frac{1}{2}dL^*_{i+1}(t) + \frac{1}{2} dL^*_{i}(t), \ t \ge 0, i \in \mathbb{N}_0.
\end{equation}
Here, $L^*_{0}(\cdot) \equiv 0$ and for $i \in \mathbb{N}$, $L^*_i(\cdot)$ denotes the local time of collision between the $(i-1)$-th and $i$-th particles, that is, the unique non-decreasing continuous process satisfying $L^*_i(0)=0$ and $L^*_i(t) = \int_0^t\mathbf{1}\left(Y_{(i-1)}(s)=Y_{(i)}(s)\right)dL^*_i(s)$ for all $t \ge 0$.
The gap process in the infinite Atlas model is the $\R_+^{\infty}$-valued process $\Z(\cdot)=$ $(Z_1(\cdot), Z_2(\cdot), \dots)$ defined by
\begin{equation}\label{gapdef}
Z_i(\cdot) := Y_{(i)}(\cdot) - Y_{(i-1)}(\cdot), \ i \in \mathbb{N}.
\end{equation}
We will be primarily interested in the long time behavior of the gap process. Thus, without loss of generality, we will assume $Y_{(0)}(0)=0$. Denote by $\cls_0$ the class of probability measures on $\RR_+^{\infty}$ such that the corresponding $\RR_+^{\infty}$-valued random variable $\YY(0) = (Y_i(0))_{i\in \NN_0}$ satisfies \eqref{S} and 
$$0 = Y_0(0)\le Y_1(0) \le \cdots \mbox{ a.s. }$$
Given a measure $\gamma \in \cls_0$, from weak existence and uniqueness, we can construct a filtered probability space $(\Om, \clf, \Pmb, \{\clf_t\}_{t\ge 0})$ on which are given mutually independent $\clf_t$-Brownian motions $\{W_i(\cdot), i \in \NN_0\}$ and $\clf_t$ adapted continuous processes $\{Y_i(t), t \ge 0, i \in \NN_0\}$ such that $Y_i$ solve \eqref{Atlasdef} with $\Pmb\circ\YY(0)^{-1} = \gamma$.
On this space the processes $Y_{(i)}$ and $Z_i$ are well defined and the distribution $\Pmb\circ \Z(\cdot)^{-1}$ is uniquely determined from $\gamma$. Furthermore $\Z$ is a Markov process with values in (a subset of) $\RR_+^{\infty}$. Let
$$\cls = \{\Pmb\circ \Z(0)^{-1}: \Pmb \circ \YY(0)^{-1} \in \cls_0\}.$$
Note that $\mu \in \mathcal{S}$ if for some $\RR_+^{\infty}$-valued random variable $\YY(0)$ with probability law in $\cls_0$,  the vector \textcolor{black}{$(Y_1(0) - Y_0(0), Y_2(0) - Y_1(0),\dots)$} has  probability law $\mu$.

It was shown in \cite{PP} that the distribution $\pi := \otimes_{i=1}^{\infty} \operatorname{Exp}(2)$ (which is clearly an element of $\cls$) is a stationary distribution of the gap process $\Z(\cdot)$. It was also conjectured there that this is the unique stationary distribution of the gap process. Surprisingly, this was shown to be not true in \cite{sarantsev2017stationary} who gave an uncountable collection of stationary distributions in $\cls$ for the gaps in the infinite Atlas model defined as
\begin{equation}\label{altdef}
\pi_a := \bigotimes_{i=1}^{\infty} \operatorname{Exp}(2 + ia), \ a \ge 0.
\end{equation}
\textcolor{black}{The `maximal' stationary distribution $\pi_0$ (in the sense of stochastic domination) is somewhat special in this collection as described below. For this reason, in what follows, we will continue using the notation $\pi$ for the measure $\pi_0$.}

For a $\nu \in \cls$, we will denote by $\hat \nu_t$ the probability law of $\Z(t)$, when $P\circ \Z(0)^{-1} = \nu$. Obviously, when $\nu = \pi_a$, $\hat\nu_t  = \pi_a$ for all $t\ge 0$. In general, one expects that under suitable conditions on $\nu \in \cls$, $\hat \nu_t$ converges (say in the weak convergence topology on the space $\clp(\RR_+^{\infty})$ of probability measures on $\RR_+^{\infty}$) to one of the stationary measures $\pi_a$, $a\ge 0$. When that happens (i.e. $\hat \nu_t \to \pi_a$ weakly as $t\to \infty$ for some $a\ge 0$), we say that $\nu$ is in the {\em Domain of Attraction} (DoA) of $\pi_a$.
Characterizing the domain of attraction of the above collection of stationary measures has been a longstanding open problem. \textcolor{black}{It was shown in \cite[Theorem 4.7]{AS} using comparison techniques with finite-dimensional  Atlas models that if the law $\nu$ of the  initial gaps $\Z(0)$ stochastically dominates $\pi = \pi_0$ (in a coordinate-wise sense), then $\nu$ is in DoA of $\pi$, i.e. the law of $\Z(t)$ converges weakly to $\pi$ as $t \rightarrow \infty$ (see also Lemma \ref{AS4.7}).} Recently, \cite{DJO} established a significantly larger domain of attraction for $\pi$ using relative entropy and Dirichlet form techniques. They showed that a $\nu \in \cls$ is in DoA of $\pi$ if the random vector $\Z(0)$ with distribution $\nu$ almost surely satisfies the following conditions for some $\beta \in [1,2)$ and eventually non-decreasing sequence $\{\theta(m) : m \ge 1\}$ with $\inf_m\{\theta(m-1)/\theta(m)\}>0$:
\begin{align}
\label{d1} &\limsup_{m \rightarrow \infty}\frac{1}{m^{\beta}\theta(m)} \sum_{j=1}^m Z_j(0) < \infty,\\
\label{d2} &\limsup_{m \rightarrow \infty}\frac{1}{m^{\beta}\theta(m)} \sum_{j=1}^m (\log Z_j(0))_- < \infty,\\ 
\label{d3} &\liminf_{m \rightarrow \infty}\frac{1}{m^{\beta^2/(1+\beta)}\theta(m)} \sum_{j=1}^m Z_j(0) = \infty,
\end{align}
with the additional requirement that $\theta(m) \ge \log m, m \ge 1,$ if $\beta = 1$. This is a big leap in our understanding of the domain of attraction properties of the infinite Atlas model. However, the conditions \eqref{d1}-\eqref{d3} involve upper and lower bounding the growth rate of the starting points $Y_m(0) = \sum_{j=1}^m Z_j(0)$ with respect to $m$ in terms of a common parameter $\beta \in [1,2)$, which partially restricts its applicability. In particular, they do not cover all initial gap distributions that stochastically dominate $\pi$ (e.g. $Z_j(0)\sim e^{j^2}$), which were shown to be in DoA of $\pi$ in \cite{AS}. Moreover, \eqref{d2} is not satisfied  if even one of the gaps is zero, which intuitively should not drastically influence ergodic properties of the model. This condition arises from the relative entropy methods used in \cite{DJO} (see e.g. the estimate  (3.5) therein) which make an important use of the fact that  the gaps are non-zero.

Beyond the results presented above, little is known about the domain of attraction of stationary measures of the infinite Atlas model. Especially, for the other stationary measures $\pi_a, a >0$, nothing is known about the domain of attraction. To investigate the latter question, the techniques of \cite{DJO} can no longer be applied in a straightforward manner, as we now explain, and one needs new ideas. The methods in \cite{DJO} involve approximating the infinite Atlas model by the finite Atlas model with $d+1$ particles, given by the solution to the SDE \eqref{Atlasdef} with $\NN_0$ replaced with $ \{0, 1, \cdots , d\}$.
A key step in the proof is to argue that the law of the first $k$ gaps, as $t\to \infty$ and $d\to \infty$ simultaneously, in a suitable fashion, converges to the $k$-marginal of $\pi$ (i.e.  $\otimes_{i=1}^{k} \operatorname{Exp}(2)$).  It is then shown by a natural coupling argument that, on any compact time interval $[0,T]$, there is a $d_T \in \NN$ such that the ranked particles in the $d_T$-dimensional Atlas model stay uniformly close to the lowest $d_T+1$ particles in the infinite Atlas model on $[0,T]$.
A crucial fact that is exploited in the proof is that the {\em unique} stationary measure of the finite Atlas model with $d+1$ particles, given by $\pi^{(d)} := \bigotimes_{i=1}^{d} \operatorname{Exp}(2(1-i/(d+1)))$, converges to $\pi$ as $d \rightarrow \infty$. Since the finite Atlas model has a unique stationary distribution, a finite-dimensional approximation approach of the form used in \cite{DJO}  cannot work in addressing convergence to $\pi_a$ for $a \neq 0,$ as this approximation is {\em designed} to select $\pi_0$. 
One may consider other types of finite-dimensional approximations to the infinite Atlas model (see Remark \ref{rem:difffd}), however they present different challenges 
 with implementing the  approach as in \cite{DJO} for showing convergence to $\pi_a$ for $a\neq 0$. These are further discussed in Remark \ref{rem:difffd}.

In the current work, we consider a somewhat weaker formulation of domain of attraction of the stationary measures $\pi_a$, $a \ge 0$.
Specifically, for $\nu \in \cls$, define for $t>0$
$$\nu_t \doteq \frac{1}{t} \int_0^t \hat \nu_s ds,$$
where $\hat \nu_t$ is as introduced in the paragraph below \eqref{altdef}.
We say that a $\nu \in \cls$ is in the {\em Weak  Domain of Attraction} (WDoA) of $\pi_a$ for $a\ge 0$, if $\nu_t \to \pi_a$ weakly as $t\to \infty$. \textcolor{black}{Note that, although the convergence of the time-averaged laws $\nu_t$  to $\pi_a$ (ergodic limit) is implied by the convergence of $\hat \nu_t$ to $\pi_a$ (marginal time limit), the converse is not clear.}
\textcolor{black}{However, if $\nu \in \cls$ is in the WDoA of $\pi_a$ for some $a>0$, it is \emph{not} in the DoA of $\pi=\pi_0$.} In this article we will provide sufficient conditions for a measure $\nu \in \cls$ to be in the WDoA of $\pi_a$ for a general $a\ge 0$.

In Theorem \ref{piconv} and Corollary \ref{starcheck}, we describe a large set of measures $\nu \in \cls$ that are  in the weak domain of attraction of $\pi = \pi_0$. The sufficient condition for $\nu$ to be in the WDoA of $\pi$ is a condition similar to \eqref{d3} with $\beta =1$ (see \eqref{star}). In particular, the sufficient condition in  \eqref{star} is implied by \eqref{d3} if the initial gap sequence $\{Z_j(0)\}$ satisfies $\limsup_{n\to \infty} Z_n(0)<\infty$ a.s. (see Remark \ref{remthm1}). We do not require upper bounds on growth rates of $m \mapsto Y_m(0) = \sum_{j=1}^m Z_j(0)$ or $m \mapsto \sum_{j=1}^m (\log Z_j(0))_-$ of the form in \eqref{d1}, \eqref{d2}.
In particular, unlike the condition in \eqref{d2}, the sufficient condition we provide does not require $Z_j(0)$ to be non-zero. Furthermore, the sufficient condition in \eqref{star} is satisfied by all gap distributions which stochastically dominate $\pi$. In Corollary \ref{starcheck} and Remark \ref{remthm1}, we provide additional examples where the sufficient condition \eqref{star} is satisfied and make some comparisons with the sufficient conditions given in \cite{DJO}.

In Theorem \ref{piaconv} and Corollary \ref{staracheck}, we provide sufficient conditions for $\nu \in \cls$ to be 
in the weak domain of attraction of the alternate stationary measures $\pi_a, a >0$. This condition 
is formulated in terms of  growth rates of the $L^1$ distance between the first $m$ coordinates of the initial gap process and that of a coupled random variable $\V_a$, with distribution $\pi_a$, as $m \rightarrow \infty$. In particular, this sufficient condition holds if the initial gap sequence is given as a perturbation $\mathbf{\Theta} = (\Theta_i)_{i\in \NN}$ of $\V_a$ which satisfies
\begin{equation}\label{condloglog}
	\sum_{i=1}^d |\Theta_i| = o\left( \frac{\log d}{\log \log d}\right) \mbox{ and } \limsup_{d\to \infty} \frac{|\Theta_d| }{dV_{a,d}} <\infty.
\end{equation}
Note that, for $a, a'>0$, the $L^1$-distance between the first $d$ coordinates of the stationary measures $\pi_a$ and $\pi_{a'}$ grows at rate $O(|a-a'| \log d)$ as $d \rightarrow \infty$ and thus the size of the perturbation allowed in \eqref{condloglog} is not very far from what one expects. See Remark \ref{remthm2} for further discussion of this point.
To the best of our knowledge, this is the first result on convergence properties for these alternate stationary measures.
In Corollary \ref{staracheck} and Remark \ref{remthm2} we provide some examples  of measures $\nu \in \cls$ in WDoA of $\pi_{a}$ for $a>0$.

Our proofs are based on a pathwise approach to studying the long time behavior of the infinite Atlas model using \emph{synchronous couplings}, namely, two versions of the infinite Atlas model described in terms of the ordered particles via \eqref{order_SDE} and started from different initial conditions but driven by the same collection of Brownian motions. Central ingredients in the proofs are suitable estimates on the decay rate of the $L^1$ distance between the gaps in synchronously coupled  ordered infinite Atlas models. These estimates are obtained by analyzing certain excursions 
of the difference of the coupled processes where each excursion ensures the contraction of the $L^1$ distance by a fixed deterministic amount. The key is to appropriately control the number and the lengths of such excursions. Certain monotonicity properties of synchronous couplings (see Proposition \ref{syncprop}), and a quantification of the influence of far away coordinates on the first few gaps (see Section \ref{farsec}), also play a crucial role. Using these tools, the discrepancy between the gap processes of two synchronously coupled infinite Atlas models can be controlled in terms of associated gap processes when the starting configurations differ only in finitely many coordinates. 
The latter are more convenient to work with as for them the 
initial $L^1$ distance between the  gap processes is finite, and the aforementioned excursion analysis can be applied to obtain our main results.

This article is organized as follows. In Section \ref{main}, we state the main results. In Section \ref{syncsec}, we collect some crucial monotonicity properties of synchronous couplings including the monotonicity and quantitative control of the $L^1$ distance between them. In Section \ref{farsec}, we quantify the influence of far away coordinates on the first $k$ gaps of the infinite Atlas model. This is crucial in reducing the problem of convergence to stationarity from arbitrary starting gap configurations to those that are perturbations of the stationary gaps at finitely many coordinates. In Section \ref{excsec}, we identify excursions in the paths of the synchronously coupled gap processes that result in reductions of the $L^1$ distance between them by fixed amounts. Finally, in Section \ref{mainproofs}, the main results are proved by analyzing these excursions of synchronously coupled gap processes from carefully chosen starting configurations.

\textbf{Notation: } For a vector $\vv = (v_1,v_2,\dots)^T \in \mathbb{R}_+^{\infty}$ and $m \in \mathbb{N}$, we will write $\vv \vert_m := (v_1,\dots,v_{m})^T$. 
We will write $s(\vv)$ for the new vector in $\mathbb{R}_+^{\infty}$ whose $i$-th coordinate is $v_1 + \dots + v_i$, for $i \in \mathbb{N}$. For two vectors $\mathbf{v}_1$ and $\mathbf{v}_2$ in $\mathbb{R}_+^{m}$, where $m \in \mathbb{N} \cup \{\infty\}$, $\mathbf{v}_1\le \mathbf{v}_2$ will be used to denote that each coordinate of $\mathbf{v}_1$ is less than or equal to the corresponding coordinate of $\mathbf{v}_2$. Similarly, $\mathbf{v}_1 \wedge \mathbf{v}_2$ and $\mathbf{v}_1 \vee \mathbf{v}_2$ will respectively denote the coordinate-wise minimum and maximum. 

\textcolor{black}{We will denote by $C([0,\infty):\RR_+^{\infty})$ the space of all continuous $\RR_+^{\infty}$-valued functions defined on $[0,\infty)$. \textcolor{black}{This space is equipped with the usual local uniform topology, namely a sequence} $f^{(m)} \in C([0,\infty):\RR_+^{\infty})$ converges to $f \in C([0,\infty):\RR_+^{\infty})$ if 
$$\lim_{m \rightarrow \infty} \sup_{t \in [0,T]}|f^{(m)}_j(t) - f_j(t)| =0 \mbox{ for any } T \in (0,\infty) \mbox{ and } j \in \NN.$$}

 In the sequel, for a collection of events $\{E_t : t \in [0,\infty)\}$, we will often write `$E_t$ holds for all $t \ge 0$' to mean `almost surely, $E_t$ holds for all $t \ge 0$'. $\Phi(\cdot)$ will denote the normal cdf and $\bar{\Phi}(\cdot) := 1- \Phi(\cdot)$. We will write $X \sim \mu$ to denote that the random variable $X$ has distribution $\mu$. For any measure $\nu$ on $\mathbb{R}_+^{\infty}$ and any $k \in \mathbb{N}$, we call the measure $\nu^{(k)}$ on $\mathbb{R}_+^{k}$ defined by $\nu^{(k)}(A) := \nu(A \times \mathbb{R}_+^{\infty}), \ A \in \mathcal{B}\left(\mathbb{R}_+^k\right),$ as the \emph{$k$-marginal} of $\nu$.
By a coupling of two probability measures $\theta_1$ and $\theta_2$ on some Polish space $S$, we mean $S$-valued random variables $X_1$, $X_2$ on some probability space $(\Om, \clf, \Pmb)$ such that $\Pmb\circ X_i^{-1} = \theta_i$ for $i=1,2$. Coupling of more than two probability measures is defined in a similar manner.
For $m \in \mathbb{N} \cup \{\infty\}$ and probability measures $\theta_1$ and $\theta_2$ on $\RR_+^m$ we say $\theta_2$ stochastically dominates $\theta_1$ if
$\theta_2(-\infty, \mathbf{x}] \le \theta_1(-\infty, \mathbf{x}]$ for all $\mathbf{x} \in \mathbb{R}^m_+$.
In this case, by Strassen's theorem \cite{lindvall1999strassen}, there is a coupling $(X, Y)$ of $(\theta_1, \theta_2)$ such that $X_i \le Y_i$ a.s. for all $i$.

\section{Main results}\label{main}
Our first theorem gives a sufficient condition for a measure on $\RR_+^{\infty}$ to be in the WDoA of $\pi$.
 \begin{theorem}\label{piconv}
 Suppose that the probability measure
 $\mu$ on $\RR_+^{\infty}$ satisfies the following: there exists a coupling $(\U,\V)$ of $\mu$ and $\pi$ such that, almost surely,
 \begin{equation}\label{star}
 \liminf_{d \rightarrow \infty} \frac{1}{\sqrt{d} (\log d)} \sum_{i=1}^d U_i \wedge V_i = \infty.
 \end{equation}
 Then $\mu \in \cls$ and it belongs to the WDoA of $\pi$.
 \end{theorem}

 The following corollary gives some natural situations in which \eqref{star} holds.
 
 \begin{corollary}\label{starcheck}
Let $\mu$ be a probability measure on $\RR_+^{\infty}$ and
  $\U \sim \mu$. Suppose one of the following conditions hold:
 \begin{itemize}
 \item[(i)] $\mu$ is stochastically dominated by $\pi$ and, almost surely,
 \begin{equation}\label{star1}
 \liminf_{d \rightarrow \infty} \frac{1}{\sqrt{d} (\log d)} \sum_{i=1}^d U_i = \infty.
 \end{equation}
 \item[(ii)] For some a.s. finite random variable $M$
 \begin{equation}\label{star2}
 \liminf_{d \rightarrow \infty} \frac{1}{\sqrt{d} (\log d)} \sum_{i=1}^d (U_i \wedge M) = \infty, \mbox{ a.s. }
 \end{equation}
 \item[(iii)] $U_i = \lambda_i \Theta_i, \ i \in \mathbb{N}$, where $\{\Theta_i\}$ are iid non-negative random variables satisfying $\mathbb{P}(\Theta_1>0)>0$, and \textcolor{black}{$\{\lambda_i\}$ are positive deterministic real numbers satisfying one of the following: 
 \begin{itemize}
 \item[(a)] $\liminf_{j \rightarrow \infty} \lambda_j>0$,\\
 \item[(b)] $\limsup_{j \rightarrow \infty}\lambda_j < \infty$ and
 \begin{equation}\label{star3}
 \liminf_{d \rightarrow \infty}\frac{1}{\sqrt{d} (\log d)} \sum_{i=1}^d \lambda_i = \infty.
 \end{equation}
 \end{itemize}}
 \end{itemize}
 Then \eqref{star} holds. Hence, in all the above cases, $\mu \in \cls$ and is in the WDoA of $\pi$.

 \end{corollary}

 \begin{remark}
	 \label{remthm1}
	 We make the following observations.
	 \begin{enumerate}[(a)]
	 	\item Suppose that $\pi$ is stochastically dominated by $\mu$. Then \eqref{star} holds and so $\mu \in \cls$ and is in the WDoA of $\pi$. \textcolor{black}{In fact in this case  \cite[Theorem 4.7]{AS} (cf. Lemma \ref{AS4.7}) shows that $\mu$ is in the DoA of $\pi$.}
		\item Note that for any $C \in (0,\infty)$, 
		$$((2C) \wedge 1) [U_i \wedge (1/2)] \le U_i \wedge C \le ((2C) \vee 1) [U_i \wedge (1/2)].$$
		 Hence, if \eqref{star2} holds for one a.s. finite random variable $M$ then it holds for every such random variable.
		 In particular if $\limsup_{n\to \infty} U_n <\infty$ a.s. and \eqref{star1} is satisfied then \eqref{star} holds and so $\mu \in \cls$ and is in the WDoA of $\pi$.
		 \item The paper \cite{DJO} notes two important settings where conditions \eqref{d1}-\eqref{d3} are satisfied. These are: (i) for some $c \in [1,\infty)$, $\lambda_j \in [c^{-1}, c]$ for all $j\in \NN$ and 
		 $ Z_j(0) = \lambda_j \Theta_j$ where $\Theta_j$ are iid with finite mean and such that 
		 $E\log (\Theta_1)_ {-}<\infty$, and (ii) $ Z_j(0) = \lambda_j \Theta_j$ where $\Theta_j$ are iid exponential  with mean $1$ and, either $\lambda_d\downarrow 0$ and $\frac{1}{\sqrt{d}\log d} \sum_{i=1}^d \lambda_i \to \infty$, as $d\to \infty$, or $\lambda_d\uparrow \infty$ and $\limsup_{d\to \infty}\frac{1}{d^{\beta}} \sum_{i=1}^d \lambda_i < \infty$ for some $\beta <2$. We note that conditions assumed in the above settings are substantially stronger than the one assumed in part (iii) of Corollary \ref{starcheck}. \textcolor{black}{However  \cite{DJO}  establishes the stronger  marginal time convergence as opposed to an ergodic limit considered in the present paper. See part (g) on a related open problem.}
		 \item In the setting of (iii) the sequence $\{\lambda_i\}$ satisfies \eqref{star3} if 
		   $\lambda_i \ge \varphi_i \log i/\sqrt{i}$ for sufficiently large $i$, for any non-negative sequence $\varphi_i \rightarrow \infty$ as $i \rightarrow \infty$.
		 Indeed, note that, for any $C>0$, there exists $i_C \in \mathbb{N}$ such that for all $i \ge i_C$,
		  $$
		  \sum_{j=i_C}^i \lambda_j \ge C\sum_{j=i_C}^i\frac{\log j}{\sqrt{j}} \ge C\int_{i_C}^{i+1}\frac{\log z}{\sqrt{z}}dz.
		  $$
		  \textcolor{black}{Using the fact that the primitive of $\log z/\sqrt{z}$ is $2\sqrt{z}(\log z - 2)$, $\{\lambda_j\}$ is seen to satisfy \eqref{star3}}
		  and so, if in addition $\limsup_{j \rightarrow \infty}\lambda_j < \infty$, we have that
		  the conditions of (iii) (b) are satisfied.
		 \item  We note that the conditions prescribed in Theorem \ref{piconv} and Corollary \ref{starcheck} all involve quantifying how close together the particles can be on the average in the initial configuration of the Atlas model. In particular, we do not require any 
		 upper bounds on the sizes of $U_i$.
		 This is in contrast with condition \eqref{d1} assumed in \cite{DJO} which requires the particles to be not too far apart on the average.  Intuitively one expects convergence to $\pi$ to hold when initially the particles are not too densely packed and upper bounds on the  average rate of growth of the initial spacings are somewhat unnatural. The result in \cite[Theorem 4.7]{AS} also points to this heuristic by showing weak convergence to $\pi$ from all initial gap configurations that stochastically dominate $\pi$. 
		 We also note that we do not require any condition analogous to \eqref{d2} and can, in particular, allow gaps to be zero.

		\item  We note that the conditions \eqref{d1}-\eqref{d3} of \cite{DJO} do not  imply our conditions. To see this, consider the deterministic sequence of initial gaps: $U_i = i^{-2/3}$ for $n^3 < i < (n+1)^3$ and $U_{n^3} = n$, for any $n \in \mathbb{N}$. It can be checked that, $\sum_{i=1}^d U_i$ grows like $d^{2/3}$ and $\sum_{i=1}^d (\log U_i)_-$ grows like $d\log d$ as $d \rightarrow \infty$. Thus,  \eqref{d1}-\eqref{d3} of \cite{DJO} hold with $\beta=1$ and $\theta(d)= \log d$. However, almost surely, $\sum_{i=1}^d U_i \wedge V_i$ grows like $d^{1/3}$, and therefore, \eqref{star} is violated.
		\item It will be interesting to investigate whether the condition in Theorem \ref{piconv} in fact  implies the stronger property  that $\mu$ is in the DoA of $\pi$. We leave this as an open problem.
	 \end{enumerate}
 
 \end{remark}
 
 Our second theorem provides sufficient conditions for a measure on $\RR_+^{\infty}$ to be in the WDoA of
 one of the other stationary measures $\pi_a, a>0$.
 
 \begin{theorem}\label{piaconv}
 Fix $a>0$. Suppose that the probability measure $\mu$ on $\RR_+^{\infty}$ satisfies the following:
 There exists a coupling $(\U, \V_a)$ of  $\mu$  and $\pi_a$ such that, almost surely,
\begin{equation}\label{stara}
 \limsup_{d \rightarrow \infty} \frac{\log \log d}{\log d} \sum_{i=1}^d |V_{a,i} - U_i| = 0, \mbox{ and } \limsup_{d\to \infty} \frac{U_d}{dV_{a,d}}<\infty.
 \end{equation}
 Then $\mu \in \cls$ and it belongs to the WDoA of $\pi_a$.
 \end{theorem}
\textcolor{black}{
We remark that Theorem \ref{piaconv} also holds for the case $a=0$. Indeed, suppose that
the condition \eqref{stara} in the theorem holds for $a=0$. Denoting the corresponding  $\V_0$ as $\V$, we have that
 $$
 \sum_{i=1}^d U_i \wedge V_i = \sum_{i=1}^d U_i \vee V_i - \sum_{i=1}^d |V_i - U_i| \ge \sum_{i=1}^d V_i - \sum_{i=1}^d |V_i - U_i|.
 $$
 Using this and the law of large numbers for $\{V_i\}$, it follows that 
  \eqref{star} holds and thus from Theorem \ref{piconv} the conclusion of Theorem \ref{piaconv} holds with $a=0$. The reason we do not note the case $a=0$ in the statement of Theorem \ref{piaconv} is because \eqref{star} is a strictly weaker assumption than \eqref{stara} when $a=0$. This is expected as for any initial gap distribution that stochastically dominates $\pi$, convergence to $\pi$ holds (see Remark \ref{remthm1} (a)). However, for convergence to $\pi_a$ for some $a>0$, the initial gap distribution should be appropriately close to $\pi_a$ (see Remark \ref{remthm2} (b) below).}
 
 The following corollary gives a set of random initial conditions for which \eqref{stara} holds.
 
 \begin{corollary}\label{staracheck}
 Fix $a>0$. Let
 $$
 \mu := \bigotimes_{i=1}^{\infty} \operatorname{Exp}(2 + ia + \lambda_i),
 $$
 where $\{\lambda_i\}$ are real numbers satisfying $\lambda_i \ge -\beta (2+ia)$ for some $\beta <1$, and $\lambda_i \le C(2+ia)$ for some $C>0$,
for all $i$. Moreover, assume
 \begin{equation}\label{aexp}
 \limsup_{d \rightarrow \infty} \frac{\log \log d}{\log d}\sum_{i=1}^d\frac{|\lambda_i|}{i^2} = 0. 
 \end{equation}
 Then \eqref{stara} holds, and hence, $\mu \in \cls$ and belongs to the WDoA of $\pi_a$. 
 \end{corollary}
 
 \begin{remark}
	 \label{remthm2}
	 We note the following.
	 \begin{enumerate}[(a)]
		 \item Clearly, there are measures $\mu$ of the form described in Corollary \ref{staracheck} for which $\liminf_{i\to \infty}\frac{|\lambda_i|}{i}>0$ and which do not lie in the WDoA of $\pi_a$ (for example, $\pi_{a'}$ for any $a' \neq a$). However, the corollary says that if $|\lambda_i|$ grows slightly slower than $i$ then $\mu$ is indeed in the WDoA of $\pi_a$. More precisely, the convergence  in \eqref{aexp} holds in particular if $|\lambda_i| \le i \delta_i/(\log \log i)$ for sufficiently large $i$, for some non-negative sequence $\delta_i \rightarrow 0$ as $i \rightarrow \infty$. 
	Indeed,  note that for any $\delta>0$, there exists sufficiently large $i_\delta \in \mathbb{N}$ such that for $i \ge i_\delta$,
\begin{equation}\label{ac2}
\sum_{j=i_\delta}^i \frac{|\lambda_j|}{j^2} \le \sum_{j=i_\delta}^i \frac{\delta_j}{j \log \log j} \le \delta\int_{i_\delta-1}^{i} \frac{1}{z\log \log z}dz = \delta\int_{\log(i_\delta-1)}^{\log i} \frac{1}{\log w}dw.
\end{equation}
As $w \mapsto \frac{1}{\log w}$ is a slowly varying function, by Karamata's theorem \cite[Theorem 1.2.6 (a)]{mikosch1999regular}, 
$$
 \lim_{i \rightarrow \infty} \frac{\log \log i}{\log i}\int_{\log(i_0-1)}^{\log i} \frac{1}{\log w}dw = 1.
 $$
 \textcolor{black}{This can also be directly shown by noting that the primitive of $1/\log w$ is $\operatorname{Li}(w)$, and applying the l'H\^{o}pital's rule.}
 Thus, we conclude from \eqref{ac2} that for any $\delta>0$,
 $$
  \limsup_{i \rightarrow \infty} \frac{\log \log i}{\log i} \sum_{j=1}^i \frac{|\lambda_j|}{j^2} \le \delta.
 $$
As $\delta>0$ is arbitrary, \eqref{aexp} holds.

 \item Note that if $\U \sim \pi$, then for $a, a'>0$,
 $$V_{a,i} = \frac{2}{2+ia} U_i, \; V_{a',i} = \frac{2}{2+ia'} U_i, \; i \in \NN$$
 defines a coupling of $\pi_a$ and $\pi_{a'}$.
 For this coupling,
 \begin{equation}\label{eq:203}
	 \sum_{i=1}^d |V_{a,i}- V_{a',i}| \sim O(|a-a'| \log d) \mbox{ and } \limsup_{d\to \infty} \frac{V_{a', d}}{dV_{a,d}} = 0.
	 \end{equation}
Since $\pi_{a'}$ is obviously not in WDoA of $\pi_a$ for $a\neq a'$, the first
 property in  \eqref{eq:203} says that the first requirement in \eqref{stara} is not far from what one expects.
  %
  %
  We conjecture that for $\mu$ to be in the WDoA of $\pi_a$  it is necessary that $(\log d)^{-1}\sum_{i=1}^d |V_{a,i} - U_i| \rightarrow 0$ as $d \rightarrow \infty$ 
 for some coupling $(\U, \V_a)$ of  $\mu$  and $\pi_a$.  
 Theorem \ref{piaconv}  says that if the first convergence holds at a rate faster than $1/(\log \log d)$ for some coupling of $\mu$  and $\pi_a$ then that (together with the second condition in \eqref{stara}) is sufficient for 
 $\mu$ to be in the WDoA of $\pi_a$.
 
Further, observe that if the second condition in \eqref{stara} does not hold, then there are infinitely many $d \in \mathbb{N}$ such that the lowest $d+1$ particles are separated from the rest by a large initial gap. If one such gap is very large, then it could happen that the distribution of gaps between the $d+1$ particles stabilizes towards the unique stationary gap distribution $\pi^{(d)}$ of the corresponding finite Atlas model before the remaining particles have interacted with them. 
As $\pi^{(d)}$ converges weakly to $\pi$ as $d \rightarrow \infty$, one does not expect in such situations the convergence (of time averaged laws) to $\pi_a$
to hold for any $a>0$.
  \end{enumerate}
 \end{remark}
 \begin{remark}
 \label{rem:difffd}	
As noted in the Introduction, the approach of \cite{DJO} using finite-dimensional approximations does not seem to have an easy extension when investigating the DoA (or WDoA) of $\pi_a$ for $a>0$. This is because the stationary distribution of the finite-dimensional Atlas model approaches $\pi = \pi_0$ as the dimension approaches infinity. Nevertheless, there is an alternative natural finite-dimensional rank-based diffusion model  that one may consider as an approximation to the infinite Atlas model which is better suited for proving convergence to $\pi_a$ for $a > 0$. Consider the SDE with $s\ge 0$ and $i \in \{0,1, \ldots, d\}$,
\begin{equation}\label{AltdSDE}
 \small dY_i(s) = \left(1- \frac{a}{d+1}\right)\mathbf{1}(Y_i(s) = Y_{(0)}(s))ds - \sum_{j=1}^d \frac{ja}{d+1}\mathbf{1}(Y_i(s) = Y_{(j)}(s))ds + d W_i(s).
 \end{equation}
 It can be checked using the form of the generator of the above diffusion that the corresponding gap process has the unique stationary distribution $\pi^{a,(d)} := \bigotimes_{i=1}^{d} \operatorname{Exp}((2+ia)(1-i/(d+1)))$, 
 which clearly converges to $\pi_a$ as $d\to \infty$. There is also a natural analog of the right-anchored rank-based diffusion of the form in \cite[Section 2]{DJO} (see equation (2.12) therein) which is associated with the above dynamics and the stationary distribution $\pi_a$. In this 
model, the bottom $d$ particles evolve as a finite Atlas model reflected off the top particle, which moves deterministically with constant velocity $-a/2$.
 Both of these Markov processes (as in the setting of \cite{DJO}) are reversible with respect to the corresponding stationary distributions and are therefore well suited for Dirichlet form methods developed in \cite{DJO}. In particular, one can obtain an estimate analogous to \cite[Proposition 1.3] {DJO} for these finite-dimensional diffusions. However, because the drift and reflection structure of the above finite rank-based diffusions is substantially different from that of the infinite Atlas model, it is not clear how to devise comparison techniques between these diffusions and the infinite Atlas model, that are crucial to both \cite{DJO} and our work.
 
 Nevertheless, the above diffusions do capture some key qualitative features that are exhibited by the bottom $d+1$ particles in the infinite Atlas model starting from configurations close (in distribution) to $\pi_a$. In particular, as proved in \cite{tsai_stat}, if the initial gaps are distributed as $\pi_a$, the Atlas processes $\{Y_{(i)}(t) + at/2 : t \ge 0\}_{i \in \mathbb{N}_0}$ (where $Y_{(i)}(\cdot)$ satisfy \eqref{order_SDE}) form a tight collection. The drift $-aj/(d+1)$ of the $j$-th ordered particle, for $j=1,\dots,d$, in the diffusion satisfying \eqref{AltdSDE} implies an analogous property for the ordered processes in this diffusion for large $d$. To see this, observe that the center of mass of the system has drift $(d+1)^{-1}\left[\left(1- \frac{a}{d+1}\right) - \sum_{j=1}^d \frac{ja}{d+1}\right]$ which is asymptotically $-a/2$ as $d \rightarrow \infty$. A similar property is true for the anchored dynamics as the top particle moves as a linear barrier with slope $-a/2$. 
Understanding relationships between these finite-dimensional rank-based diffusions and the infinite Atlas model, particularly over long time intervals, is key to extending the techniques of \cite{DJO} to study the DoA of $\pi_a$ for $a>0$. We hope to explore this in  future work. 
 \end{remark}
 
 \section{Synchronous couplings of the infinite ordered Atlas model}\label{syncsec}
 As mentioned earlier, a synchronous coupling of the ranked particles in the infinite Atlas model refers to coupling two copies of the order statistics processes starting from different initial configurations but driven by the same set of Brownian motions via the SDE \eqref{order_SDE}. The finite-dimensional analogue of the SDE \eqref{order_SDE} is known to have a unique strong solution \cite{HR}, which implies the existence of synchronous couplings for the ranked processes of finite Atlas models starting from any two initial configurations. 
\textcolor{black}{For the infinite Atlas model, devising synchronous couplings is a more delicate issue as described below.}
 
 In \cite{AS}, a systematic method for constructing solutions of the infinite-dimensional SDE \eqref{order_SDE}, from a given collection of Brownian motions, has been devised under the name of \emph{approximative versions}. It is based on taking appropriate limits of solutions of the corresponding finite-dimensional systems. This, in turn, can be used to construct synchronous couplings in an appropriate sense. We collect some results from \cite{AS} in the following proposition.
 
 \begin{proposition}\label{syncprop}
Suppose $\mathbf{x}$ is a $\mathbb{R}^{\infty}_+$-valued random variable satisfying $0 \le x_0 \le x_1 \le x_2 \le \cdots$ and, almost surely,
\begin{equation}\label{eq:initinteg}
	\sum_{i=0}^{\infty} e^{-\alpha x_i^2} < \infty, \mbox{ for all } \alpha >0.
\end{equation} 
Consider a collection of independent standard Brownian motions $B_0(\cdot),$ $B_1(\cdot), B_2(\cdot), \dots$, independent of $\mathbf{x}$. Consider the collection of 
SDE, for $i \in \NN_0$,
 \begin{equation}\label{order_SDEX}
dX_{i}(t) = \mathbf{1}(i=0) dt + dB_i(t) - \frac{1}{2}dL_{i+1}(t) + \frac{1}{2} dL_{i}(t), \ X_i(0) = x_i, \ t \ge 0, 
\end{equation}
where $L_0(\cdot) \equiv 0$, and for $i \in \mathbb{N}$, $L_i(\cdot)$ is the associated local time of collision between the $(i-1)$-th and $i$-th particle, namely
 \begin{equation}\label{eq:loctimecond}
	 \begin{aligned}
 	L_i, X_i & \mbox{ are nonnegative continuous processes, } L_i \mbox{ are nondecreasing, } \\
	L_i(0)&=0 \mbox{ and } L_i(t) = \int_0^t\mathbf{1}\left(X_{(i-1)}(s)=X_{(i)}(s)\right)dL_i(s) \mbox{ for all } t \ge 0.
	\end{aligned}
 \end{equation}
 Also consider for fixed $m \in \NN$, the system of SDE in \eqref{order_SDEX} for $i=0, 1, \ldots, m$, with 
 starting configuration $X_i(0) = x_i$, $0 \le i \le m$, and with $L_i$ satisfying 
\eqref {eq:loctimecond} for $1\le i \le m$ and $L_{0}(\cdot) \equiv L_{m+1}(\cdot) \equiv 0$.
Denote by 
 $\XX^{(m)}(\cdot) =(X^{(m)}_0(\cdot),\dots, X^{(m)}_m(\cdot))$ and $\Lb^{(m)}(\cdot) =(L^{(m)}_0(\cdot),\dots, L^{(m)}_m(\cdot))$
 the unique strong solution to this finite-dimensional system of reflected SDE which we call the `finite ordered Atlas model with $m+1$ particles' with driving Brownian motion $(B_0(\cdot),\dots,B_m(\cdot))$. Then, the following hold.
 \begin{itemize}
 \item[(i)] There exist continuous $\RR_+^{\infty}$-valued processes $\XX(\cdot) := (X_i(\cdot) : i \in \mathbb{N}_0)$,
 $\Lb(\cdot) := (L_i(\cdot) : i \in \mathbb{N}_0)$, adapted to $\clf_t \doteq \sigma \{B_i(s): s\le t, i \in \NN_0\}$, such that, a.s. $(\XX, \Lb)$ satisfy \eqref{order_SDEX} and \eqref{eq:loctimecond} and for any $T\in (0,\infty)$,
  $$
 \lim_{m \rightarrow \infty} \sup_{t \in [0,T]}\left[\left| X^{(m)}_i(t) - X_i(t)\right| +  \left| L^{(m)}_i(t) - L_i(t)\right|\right] = 0 \ \text{ a.s., for all } i \in \mathbb{N}_0.
 $$
 We will call $\XX(\cdot)$ the `infinite ordered Atlas model' with driving Brownian motions $B_0(\cdot), B_1(\cdot), B_2(\cdot), \dots,$ started from $\mathbf{x} = (x_0,x_1, x_2\dots)$.
 
\item[(ii)]  Suppose $\mathbf{u}, \mathbf{v}$ are  $\mathbb{R}_+^{\infty} $-valued random variables independent of $\{B_i, i \in \NN_0\}$, such that
 $u_i \le u_{i+1}$ and $v_i \le v_{i+1}$ a.s. for all $i \in \mathbb{N}_0$, and \eqref{eq:initinteg} holds a.s. with $\mathbf{x}$ replaced with $\mathbf{u}$, $\mathbf{v}$. Then, on $(\Om, \clf, \PP)$ there are continuous, $\RR_+^{\infty}$-valued processes  $\XX^{\mathbf{u}}, \Lb^{\mathbf{u}}$ and $\XX^{\mathbf{v}}, \Lb^{\mathbf{v}}$ that are adapted to 
 $\hat \clf_t = \clf_t \vee \sigma\{\mathbf{u}, \mathbf{v}\}$  and solve \eqref{order_SDEX}, \eqref{eq:loctimecond} 
 with $\XX^{\mathbf{u}}(0)= \mathbf{u}$ and $\XX^{\mathbf{v}}(0)= \mathbf{v}$ respectively. These processes are given as follows.
 Let $\XX^{(m),\mathbf{u}}(\cdot)$ and $\XX^{(m),\mathbf{v}}(\cdot)$ be the unique strong solutions of the finite ordered Atlas model with $m+1$ particles, with
 $\XX^{(m),\mathbf{u}}(0) = \mathbf{u} \vert_{m+1}$, $\XX^{(m),\mathbf{v}}(0) = \mathbf{v} \vert_{m+1}$.
  Then $\XX^{\mathbf{u}}$, $\XX^{\mathbf{v}}$ are defined as the a.s. limits of $\XX^{(m),\mathbf{u}}(\cdot)$ and $\XX^{(m),\mathbf{v}}(\cdot)$ as $C([0, \infty): \RR_+^{\infty})$-valued random variables. 
We will refer to these limit processes as a synchronous coupling of the infinite ordered Atlas model driven by Brownian motions $\{B_i, i \in \NN_0\}$ with initial configuration $\XX^{\mathbf{u}}(0) = \mathbf{u}$ and $\XX^{\mathbf{v}}(0) = \mathbf{v}$.
  
 
 \item[(iii)] If  
 $(\mathbf{u}, \mathbf{v})$ in (ii) are such that $\mathbf{u} \le \mathbf{v}$, then the synchronous coupling defined in (ii) almost surely satisfies $\XX^{\mathbf{u}}(t) \le \XX^{\mathbf{v}}(t)$ for all $t \ge 0$. 
 
 Define the corresponding `gap processes' \textcolor{black}{$\Z^{\mathbf{u}}(\cdot)$ and $\Z^{\mathbf{v}}(\cdot)$ by $Z^{\mathbf{u}}_i(\cdot) := X^{\mathbf{u}}_i(\cdot) - X^{\mathbf{u}}_{i-1}(\cdot)$ and $Z^{\mathbf{v}}_i(\cdot) := X^{\mathbf{v}}_i(\cdot) - X^{\mathbf{v}}_{i-1}(\cdot)$} for $i \in \mathbb{N}$. 
 Then $\Z^{\mathbf{u}}(0) \le \Z^{\mathbf{v}}(0)$ almost surely implies the following hold almost surely:
 $$
 \Z^{\mathbf{u}}(t) \le \Z^{\mathbf{v}}(t) \text{ for all } t \ge 0, \text{ and } \mathbf{L}^{\mathbf{v}}(t) - \mathbf{L}^{\mathbf{u}}(t) \le \mathbf{L}^{\mathbf{v}}(s) - \mathbf{L}^{\mathbf{u}}(s) \text{ for all } 0 \le s \le t.
 $$
 
 \item[(iv)] Analogous statements hold with $\mathbf{1}(i=0)$ in  \eqref{order_SDEX} replaced with 
$\gamma\mathbf{1}(i=0)$ where $\gamma \in \mathbb{R}$. Specifically, for
  $\gamma \in \mathbb{R}$, define the `infinite ordered Atlas model with drift $\gamma$' to be the $\mathbb{R}^{\infty}$-valued stochastic process $\XX^{[\gamma]}(\cdot) := (X^{[\gamma]}_0(\cdot),X^{[\gamma]}_1(\cdot)\dots)$ satisfying the SDE
 \begin{equation}\label{order_SDEgamma}
dX^{[\gamma]}_i(t) = \gamma \mathbf{1}(i=0) dt + dB_i(t) - \frac{1}{2}dL^{[\gamma]}_{i+1}(t) + \frac{1}{2} dL^{[\gamma]}_{i}(t), \ t \ge 0, i \in \mathbb{N}_0,
\end{equation}
where $L^{[\gamma]}_0(\cdot) \equiv 0$ and $(L^{[\gamma]}_1(\cdot), L^{[\gamma]}_2(\cdot),\dots)$ are the associated local times as in \eqref{eq:loctimecond}. Then $\XX^{[\gamma]}(\cdot)$ can be obtained as a limit of the solutions to the finite-dimensional analogues of the SDE \eqref{order_SDEgamma} in the sense of (i). In particular, one can construct a synchronous coupling of $\XX(\cdot)$ and $\XX^{[\gamma]}(\cdot)$ for same or different (coupled) initial configurations as in (ii). Moreover, if $\gamma \le 1$ and $\Z(\cdot)$ and $\Z^{[\gamma]}(\cdot)$ are the associated gap processes under synchronous coupling and $\Z(0) = \Z^{[\gamma]}(0)$, then, almost surely,
$$
\Z(t) \le \Z^{[\gamma]}(t) \ \text{ for all } \ t \ge 0.
$$
 \end{itemize}
 \end{proposition}
 
 \begin{proof}
(i) follows from \cite[Definition 7 and Theorem 3.7]{AS}, \cite[Lemma 6.4]{AS} and the discussion following it.
Part (ii) is immediate from (i).
(iii) follows from \cite[Corollary 3.10]{AS}. The assertion in (iv) that $\XX^{\gamma}(\cdot)$ can be obtained as a limit of finite systems again follows from \cite[Definition 7 and Theorem 3.7]{AS}. The pathwise domination assertion under synchronous coupling follows from \cite[Corollary 3.12 (ii)]{AS}.
 \end{proof}
 
 Proposition \ref{syncprop} constructs the infinite ordered Atlas model via approximation by finite ordered Atlas models. As the solution to the finite-dimensional analogue of \eqref{order_SDEX} is unique in law (which follows from pathwise uniqueness), the finite ordered Atlas model has the same law as the ranked processes in the finite Atlas model (which is a weak solution to the finite-dimensional analogue of \eqref{Atlasdef}). The following lemma says that the infinite ordered Atlas model, constructed in  Proposition \ref{syncprop} (i), indeed has the same law as the ranked trajectories of the infinite Atlas model satisfying \eqref{Atlasdef}. A similar result was discussed in \cite[Remark 2]{AS}. We provide a proof for completeness.
 
 \begin{lemma}\label{samelaw}
	 Let $\mathbf{x}$ be a $\RR_+^{\infty}$-valued random variable \textcolor{black}{satisfying $0 \le x_0 \le x_1 \le x_2 \le \cdots$ such that, almost surely, \eqref{eq:initinteg} holds}, and let  $\{B_i, i \in \NN_0\}$ be a sequence of independent standard Brownian motions, independent of $\mathbf{x}$. Let $\XX$ be the process constructed in Proposition \ref{syncprop} (i) with $\XX(0) = \mathbf{x}$. Let $\mu$ be the \textcolor{black}{probability law of $\mathbf{x}$}
	 and $\YY(\cdot) := (Y_0(\cdot), Y_1(\cdot),\dots)$ be the unique weak solution of \eqref{Atlasdef} with $\YY(0) \sim  \mu$. Then $\XX$ has the same law as the
	 ranked processes $\YY_{r}(\cdot) :=(Y_{(0)}(\cdot), Y_{(1)}(\cdot), \dots)$.
\end{lemma}
 
 \begin{proof}
Let $\YY$ be as in the statement of the lemma and 
for any $m \in \mathbb{N}$, let $\YY^{(m)}(\cdot) =(Y^{(m)}_0(\cdot),\dots, Y^{(m)}_m(\cdot))$ be the unique weak solution of  the
finite Atlas model \eqref{Atlasdef}  with $m+1$ particles
(i.e. with $i \in \NN_0$ in \eqref{Atlasdef} replaced by $0 \le i \le m$)
and with $\YY^{(m)}(0)$ distributed same as $(Y_0(0), Y_1(0), \ldots Y_m(0))$.
Let   $\YY^{(m)}_{r}(\cdot) :=(Y^{(m)}_{(0)}(\cdot),\dots, Y^{(m)}_{(m)}(\cdot))$ be the corresponding ranked trajectories. By \cite[Theorem 3.3]{AS}, there exists a sequence $m_j \rightarrow \infty$ such that for any $T >0$ and $k \in \mathbb{N}$,
\begin{equation}\label{eq1}
(\YY^{(m_j)}(\cdot)\vert_k, \YY^{(m_j)}_{r}(\cdot)\vert_k) \xrightarrow{d} (\YY(\cdot)\vert_k, \YY_{r}(\cdot)\vert_k) \ \text{ on } \mathbb{C}([0,T] : \mathbb{R}^{2k}).
\end{equation}
Next let $\mathbf{x}$, $\{B_i, i \in \NN_0\}$ be  as in the statement of the lemma and let $(\XX^{(m)}, \Lb^{(m)})$ be as in Proposition \ref{syncprop} (i) given as the unique strong solution of the corresponding $(m+1)$-dimensional reflected SDE, with $\XX^{(m)}(0) = \mathbf{x}\vert_{m+1}$.
%
\textcolor{black}{From the same argument used to derive \eqref{order_SDE} from \eqref{Atlasdef}, it follows that $\YY^{(m_j)}_{r}(\cdot)$ satisfies the $m_j$-dimensional version of the SDE \eqref{order_SDE} for each $j \ge 1$.
By pathwise uniqueness of solutions to the finite-dimensional reflected SDE \cite{HR}, for any $k \in \mathbb{N}$,}
\begin{equation}\label{eq2}
\YY^{(m_j)}_{r}(\cdot)\vert_k \equald \XX^{(m_j)}(\cdot) \vert_k \ \text{ for all } j \ge 1.
\end{equation}
Moreover, by Proposition \ref{syncprop} (ii), $\XX^{(m)}\vert_k \rightarrow \XX\vert_k$, almost surely, in $C([0, \infty): \RR_+^{k+1})$,
 as $m \rightarrow \infty$. This observation, along with \eqref{eq1} and \eqref{eq2}, gives
$$
\YY_{r}(\cdot)\vert_k \equald \XX(\cdot)\vert_k \ \text{ for any } k \in \mathbb{N},
$$
which proves the lemma.
 \end{proof}
 
A pathwise analysis of synchronous couplings will be the main tool in proving our main results. A key observation is the following control on the $L^1$ distance between the gap processes corresponding to two synchronously coupled infinite ordered Atlas models. In particular, it shows (using Proposition \ref{syncprop} (iii)) that if the initial gap distributions satisfy a coordinate-wise ordering, then this $L^1$ distance is non-increasing in time and its decay rate can be precisely estimated from the local times of collision between the bottom two particles. We note  that the proof for an analogous result for the finite ordered Atlas model is easier and can be deduced from standard properties of reflected Brownian motions \cite[Theorem 6]{KW}. 

\begin{lemma}\label{l1control}
Let $B_0(\cdot), B_1(\cdot), B_2(\cdot), \dots$ be standard Brownian motions given on some probability space $(\Om, \clf, \PP)$.
Let $\XX(\cdot)$ and $\tXX(\cdot)$ be  synchronously coupled infinite ordered Atlas models driven by $\{B_i, i \in \NN_0\}$ as constructed  in Proposition \ref{syncprop} (ii), with the initial configurations $\XX(0)$ and $\tXX(0)$  such that the corresponding probability laws are in $\cls_0$. 
Let $\Z(\cdot)$ and $\tZ(\cdot)$ denote the corresponding gap processes. Suppose $\tZ(0) \ge \Z(0)$. Write $\Delta \Z(\cdot) := \tZ(\cdot) - \Z(\cdot)$ and assume $\sum_{j=1}^{\infty} j \Delta Z_j(0) < \infty$. Let $\mathbf{L}(\cdot)$ and $\widetilde{\mathbf{L}}(\cdot)$ denote the local times of collision for the respective Atlas models and write $\Delta \mathbf{L}(\cdot) := \widetilde{\mathbf{L}}(\cdot) - \mathbf{L}(\cdot)$. Then, almost surely,
\begin{equation}\label{decayl1}
\sum_{j=1}^{\infty} \Delta Z_j(t) = \sum_{j=1}^{\infty} \Delta Z_j(0) + \frac{1}{2} \Delta L_1(t), \ t \ge 0.
\end{equation}
\end{lemma}

\begin{proof}
Observe form \eqref{order_SDEX} that for any $d \in \mathbb{N}, t \ge 0$,
\begin{align}\label{l0}
\sum_{j=1}^d \Delta Z_j(t) &= \sum_{j=1}^d \Delta Z_j(0) + \sum_{j=1}^d \left(\Delta L_j(t) - \frac{1}{2}\Delta L_{j-1}(t) - \frac{1}{2}\Delta L_{j+1}(t)\right)\\
&= \sum_{j=1}^d \Delta Z_j(0) + \frac{1}{2} \Delta L_1(t) + \frac{1}{2} \Delta L_d(t) - \frac{1}{2} \Delta L_{d+1}(t).\notag
\end{align}
Thus, to show \eqref{decayl1}, it suffices to show that, for every $t \ge 0$, 
\begin{equation}\label{locdec}
\lim_{d \rightarrow \infty} \ \frac{1}{2} \Delta L_d(t) - \frac{1}{2} \Delta L_{d+1}(t) = 0.
\end{equation}
This will be shown in several steps as follows.

For $m \in \mathbb{N}$, denote by $\Z^{(m)}(\cdot)$ and $\tZ^{(m)}(\cdot)$ the $m$-dimensional gap processes associated with 
ranked processes $\XX^{(m)}(\cdot)$ and $\tXX^{(m)}(\cdot)$ for the
finite ordered Atlas models with $m+1$ particles, respectively starting from $\XX(0) \vert_{m+1}$ and $\tXX(0) \vert_{m+1}$ and driven by $\{B_i\}_{i=0}^m$. Recall from Proposition \ref{syncprop} that  $\Z^{(m)}(\cdot)$ and $\tZ^{(m)}(\cdot)$ respectively converge to $\Z(\cdot)$ and $\tZ(\cdot)$ a.s. in $C([0,\infty):\RR_+^{\infty})$ as $m \rightarrow \infty$. Write $\Delta \Z^{(m)}(\cdot) := \tZ^{(m)}(\cdot) - \Z^{(m)}(\cdot)$. Similarly, let $\Delta \mathbf{L}^{(m)}(\cdot) := \widetilde{\mathbf{L}}^{(m)}(\cdot) - \mathbf{L}^{(m)}(\cdot)$ denote the difference vector between the respective local times. Then, from the SDE for $\XX^{(m)}(\cdot)$ and $\tXX^{(m)}(\cdot)$
(i.e. the finite-dimensional version of \eqref{order_SDEX}), we obtain
\begin{equation}\label{l1}
\Delta \Z^{(m)}(t) = \Delta \Z^{(m)}(0) + R^{(m)} \Delta \mathbf{L}^{(m)}(t), \ t \ge 0,
\end{equation}
where $R^{(m)}$ is an $m \times m$ matrix given by $R^{(m)}_{ii} =1$ for all $1 \le i \le m$, $R^{(m)}_{i (i-1)} = - 1/2$ for $2 \le i \le m$, $R^{(m)}_{i (i+1)} = -1/2$ for $1 \le i \le m-1$ and $R^{(m)}_{ij} = 0$ if $|i-j| \ge 2$. It can be checked (see, for example, proof of \cite[Theorem 4]{BanBudh}) that $R^{(m)}$ is invertible and
\begin{equation}\label{rmdef}
(R^{(m)})^{-1}_{ij} = 2(i \wedge j)\left(1- \frac{i \vee j}{m+1}\right), \ 1\le i,j \le m.
\end{equation}
Also, from  \cite[Theorem 6]{KW},
\begin{equation}\label{eqfromKW}
	\mbox{ For every } t\ge 0, \tZ^{(m)}(t) \ge \Z^{(m)}(t) \mbox{ and } -\Delta \mathbf{L}^{(m)}(\cdot) \mbox{ is nonnegative and nondecreasing, a.s.}
\end{equation}
Write $\zeta := \sum_{j=1}^{\infty} j \Delta Z_j(0) < \infty$. Then, it follows from \eqref{l1} that, for any $1 \le d \le m$, $t \ge 0$,
\begin{align}\label{l2}
\sum_{j=1}^d(R^{(m)})^{-1}_{dj}\Delta Z^{(m)}_j(t) &\le \sum_{j=1}^m(R^{(m)})^{-1}_{dj}\Delta Z^{(m)}_j(t)  = \sum_{j=1}^m(R^{(m)})^{-1}_{dj}\Delta Z^{(m)}_j(0) + \Delta L^{(m)}_d(t)\\
& \le  \sum_{j=1}^m(R^{(m)})^{-1}_{dj}\Delta Z^{(m)}_j(0) \le 2\sum_{j=1}^m j \Delta Z^{(m)}_j(0) = 2\sum_{j=1}^m j \Delta Z_j(0) \le 2\zeta,\notag
\end{align}
where we used the non-negativity of \textcolor{black}{$(R^{(m)})^{-1}_{dj}$} (see \eqref{rmdef}) and the first part of \eqref{eqfromKW}
for the first inequality,  the second part of \eqref{eqfromKW} for the second inequality and the explicit form of $(R^{(m)})^{-1}$ given in \eqref{rmdef} in the third inequality. Recall that, for any $t \ge 0, j \in \mathbb{N}$, $\Delta Z^{(m)}_j(t) \rightarrow \Delta Z_j(t)$ as $m \rightarrow \infty$. Moreover, using \eqref{rmdef}, for any fixed $j,d \in \mathbb{N}$ with $j \le d$, $(R^{(m)})^{-1}_{dj} \rightarrow 2j$ as $m \rightarrow \infty$. Thus, fixing $d \in \mathbb{N}$, we obtain by letting $m \rightarrow \infty$ on the left hand side of \eqref{l2}, for all $d\in \NN$,
\begin{equation}\label{l3}
2\sum_{j=1}^d j \Delta Z_j(t) \le 2\zeta, \ t \ge 0.
\end{equation}
By Proposition \ref{syncprop} (iii), the left hand side of \eqref{l3} is non-decreasing in $d$. Hence, taking  limit as $d \rightarrow \infty$,
\begin{equation}\label{l4}
\sum_{j=1}^{\infty} j \Delta Z_j(t) \le \zeta, \ t \ge 0.
\end{equation}
Next, we claim that for any $d \in \mathbb{N}, t \ge 0$,
\begin{equation}\label{l5}
0 \le - \Delta L_d(t) \le 2\zeta.
\end{equation}
To show this, we note that by  \eqref{eqfromKW} and  \eqref{l2}, for any $m \ge d$,
$$
- \Delta L^{(m)}_d(t) \le \sum_{j=1}^m(R^{(m)})^{-1}_{dj}\Delta Z^{(m)}_j(0)  \le 2\zeta.
$$
The upper bound in \eqref{l5} now follows upon taking a limit as $m \rightarrow \infty$ on the left hand side above and using Proposition \ref{syncprop} (i). The lower bound follows from Proposition \ref{syncprop} (iii).

Finally, we claim that for any $d \in \mathbb{N}, t \ge 0$,
\begin{equation}\label{l6}
-\zeta \le \frac{d+1}{2} \Delta L_d(t) - \frac{d}{2} \Delta L_{d+1}(t) \le \zeta.
\end{equation}
To see this, observe that
\begin{align*}
\sum_{j=1}^{d} j \Delta Z_j(t) &= \sum_{j=1}^{d} j \Delta Z_j(0) + \sum_{j=1}^d j \left(\Delta L_j(t) - \frac{1}{2}\Delta L_{j-1}(t) - \frac{1}{2}\Delta L_{j+1}(t)\right)\\
&= \sum_{j=1}^{d} j \Delta Z_j(0) + \frac{d+1}{2} \Delta L_d(t) - \frac{d}{2} \Delta L_{d+1}(t).
\end{align*}
Hence,
$$
\frac{d+1}{2} \Delta L_d(t) - \frac{d}{2} \Delta L_{d+1}(t) = \sum_{j=1}^{d} j \Delta Z_j(t) - \sum_{j=1}^{d} j \Delta Z_j(0).
$$
The inequality in \eqref{l6} is now immediate from the above upon using Proposition \ref{syncprop} (iii) and \eqref{l4} and noting that
$$
-\zeta = - \sum_{j=1}^{\infty} j \Delta Z_j(0) \le \sum_{j=1}^{d} j \Delta Z_j(t) - \sum_{j=1}^{d} j \Delta Z_j(0) \le \sum_{j=1}^{\infty} j \Delta Z_j(t) \le \zeta.
$$
The convergence in \eqref{locdec} now follows from \eqref{l5} and \eqref{l6} as for every $t \ge 0$,
$$
\lim_{d \rightarrow \infty} \frac{1}{2} \Delta L_d(t) - \frac{1}{2} \Delta L_{d+1}(t) = \lim_{d \rightarrow \infty} \left[\frac{1}{d}\left(\frac{d+1}{2} \Delta L_d(t) - \frac{d}{2} \Delta L_{d+1}(t)\right) - \frac{1}{2d} \Delta L_d(t)\right] = 0.
$$
This proves \eqref{decayl1} and hence the lemma.
\end{proof}

\section{Influence of far away coordinates}\label{farsec}

A crucial component of our approach will be to obtain a quantitative understanding of the time taken for the first $k$ gaps between the ranked particles in the infinite Atlas model to `feel the effect' of the unordered processes $Y_i(\cdot)$ for \textcolor{black}{$i \gg k$}. In this section, we derive such estimates. 
We recall that the infinite Atlas model is given as the unique weak solution of \eqref{Atlasdef} when the law of $Y(0)$ is an element of $\cls_0$. Note that the latter property, which will be assumed throughout this section, says that
$Y_{(0)}(0) = 0$ and $Y_{(i)}(0) = Y_i(0)$ for all $i \in \mathbb{N}_0$.

We will rely on the following estimates derived in \cite{DJO}.

\begin{lemma}\label{DJOest}
(i) (\cite[Lemma 3.2]{DJO}) For any $k \in \mathbb{N},  t >0, l \in \mathbb{N}$, $\Gamma \in \mathbb{R}$, a.s., 
\begin{align}\label{sup}
\mathbb{P}\left( \sup_{s \in [0, t]} Y_{(k)}(s) \ge \Gamma \ \big| \ \Y(0)\right) &\le 2\bar{\Phi}\left(\frac{l(\Gamma - Y_{k}(0))/3 - t - \sum_{j=0}^{l-1}Y_{j}(0)}{\sqrt{lt}} \right)\\
 &\qquad + 4(k+1)\bar{\Phi}\left(\frac{\Gamma - Y_{k}(0)}{3\sqrt{t}} \right).\notag
\end{align}
(ii) (\cite[Lemma 3.3]{DJO}) For any $d \in \mathbb{N}, t>0, \Gamma >0$, a.s., 
\begin{equation}\label{inf}
\mathbb{P}\left( \inf_{s \in [0, t]} \inf_{i \ge d}Y_{i}(s) \le \Gamma \ \big| \ \Y(0)\right) \le 2\sum_{i \ge d}\bar{\Phi}\left(\frac{Y_{i}(0) - \Gamma}{\sqrt{t}}\right).
\end{equation}
\end{lemma}

The following lemma will be used in Section \ref{mainproofs} (see proof of Lemma \ref{undconv}) to quantify influence of far away coordinates when the starting configuration satisfies \eqref{star}.

\begin{lemma}\label{farpi}
Suppose that
the distribution of 
$Y_i(0) - Y_{i-1}(0)$ is stochastically dominated by the $\operatorname{Exp}(2)$ distribution  for each $i \in \mathbb{N}$. Moreover, assume
\begin{equation}\label{farpicond}
\liminf_{d \rightarrow \infty} \frac{Y_d(0)}{\sqrt{d}(\log d)} = \infty \ \text{ almost surely.}
\end{equation}
Then for any $A \ge 4$, $k \in \mathbb{N}$, writing $t^A_d := Ad (\log d)$ and $\Gamma^A_d := 6A\sqrt{d}(\log d)$ for $d \in \mathbb{N}$,
\begin{align*}
\mathbb{P}\left(\inf_{s \in [0, t^A_d]} \inf_{i \ge d}Y_{i}(s) \le \Gamma^A_d\right) &\rightarrow 0 \text{ as } d \rightarrow \infty,\\
\mathbb{P}\left(\sup_{s \in [0, t^A_d]} Y_{(k)}(s) \ge \Gamma^A_d\right) &\rightarrow 0 \text{ as } d \rightarrow \infty.
\end{align*}
\end{lemma}

\begin{proof}
For $d \in \mathbb{N}$, let $l_d := \lceil\sqrt{d}\rceil$. Define the event
$$
\mathcal{A}(n) := \{Y_i(0) - Y_{i-1}(0) \le \log i \ \text{ for all } i > n\}, \ n \in \mathbb{N}.
$$
By the stochastic domination assumption,
\begin{align}\label{f1}
\mathbb{P}\left(\mathcal{A}(n)^c\right) \le \sum_{i=n+1}^{\infty}\mathbb{P}\left(Y_i(0) - Y_{i-1}(0) > \log i\right)
 \le \sum_{i=n+1}^{\infty} i^{-2} \le  n^{-1} \rightarrow 0 \text{ as } n \rightarrow \infty.
\end{align}
For fixed $n \in \NN$ and $d$ sufficiently large, we have
\begin{align*}
	\sum_{j=0}^{l_d-1} Y_j(0) &=  \sum_{j=0}^{n-1} Y_j(0) + \sum_{j=n}^{l_d-1} Y_j(0) =
	\sum_{j=0}^{n-1} Y_j(0) + Y_n(0) (l_d-n) + \sum_{j=n}^{l_d-1} (Y_j(0)-Y_n(0))\\
	&\le Y_n(0)l_d + \sum_{j=n}^{l_d-1} (Y_j(0)-Y_n(0)).
\end{align*}
Also, for $n\le j \le l_d-1$, on the event $\mathcal{\A}(n)$,
$$Y_j(0)-Y_n(0) \textcolor{black}{= \sum_{k=n+1}^j (Y_k(0)- Y_{k-1}(0))} \le \sum_{k=n+1}^j \log k.$$
Thus, on $\mathcal{\A}(n)$, for $d$ sufficiently large,
$$\sum_{j=0}^{l_d-1} Y_j(0) \le Y_n(0)l_d + \sum_{k=n+1}^{l_d-1} (k-1) \log k \le 
(1+ \sqrt{d}) Y_n(0) +
(1+ \sqrt{d})^2\log (1 + \sqrt{d}).$$
Thus, for fixed $n \in \NN$, there exists deterministic $d_0 \in \mathbb{N}$ such that, on the event $\mathcal{\A}(n)$, for all $d \ge d_0$,
\begin{align*}
&\frac{l_d(\Gamma^A_d - Y_{k}(0))/3 - t^A_d - \sum_{j=0}^{l_d-1}Y_{j}(0)}{\sqrt{l_d t^A_d}}\\
&\ge \frac{2Ad(\log d) - A d (\log d) - (1+ \sqrt{d})^2\log (1 + \sqrt{d})}{[(1+ \sqrt{d})Ad (\log d)]^{1/2}}
 - \frac{(1 + \sqrt{d})\left(Y_k(0)/3 + Y_n(0)\right)}{[\sqrt{d}Ad (\log d)]^{1/2}}\\
&\ge \frac{A d (\log d)}{4\sqrt{A} d^{3/4} (\log d)^{1/2}} - \frac{2\left(Y_k(0)/3 + Y_n(0)\right)d^{1/4}}{\sqrt{A d \log d}}\\
&= \frac{\sqrt{A}}{4}d^{1/4}(\log d)^{1/2} - \frac{2\left(Y_k(0)/3 + Y_n(0)\right)}{\sqrt{A} d^{1/4} (\log d)^{1/2}}.
\end{align*}
Also, for any $d \in \NN$, 
$$
\frac{\Gamma^A_d - Y_{k}(0)}{3\sqrt{t^A_d}} = \frac{2A \sqrt{d}(\log d)}{\sqrt{A d (\log d)}} - \frac{Y_k(0)}{3\sqrt{A d (\log d)}} = 2\sqrt{A \log d} - \frac{Y_k(0)}{3\sqrt{A d (\log d)}}.
$$
Thus, using \eqref{sup}, for any $d \ge d_0$, on the event $\mathcal{\A}(n)$,
\begin{align*}
\mathbb{P}\left( \sup_{s \in [0, t^A_d]} Y_{(k)}(s) \ge \Gamma^A_d  \ \vert \ \Y(0)\right) 
&\le 2\bar{\Phi}\left(\frac{\sqrt{A}}{4}d^{1/4}(\log d)^{1/2}  - \frac{2\left(Y_k(0)/3 + Y_n(0)\right)}{\sqrt{A} d^{1/4} (\log d)^{1/2}}\right)\\
&\qquad + 4(k+1) \bar{\Phi}\left(2\sqrt{A \log d} - \frac{Y_k(0)}{3\sqrt{A d (\log d)}}\right).\notag
\end{align*}
Note that the right side converges to $0$ as $d\to \infty$.
Hence, for any $n \in \mathbb{N}$,
\begin{equation*}
\limsup_{d \rightarrow \infty}\mathbb{P}\left( \sup_{s \in [0, t^A_d]} Y_{(k)}(s) \ge \Gamma^A_d\right) \le \mathbb{P}\left(\mathcal{A}(n)^c\right).
\end{equation*}
As $n \in \mathbb{N}$ is arbitrary, we conclude from \eqref{f1} that
\begin{equation}\label{f2}
\limsup_{d \rightarrow \infty}\mathbb{P}\left( \sup_{s \in [0, t^A_d]} Y_{(k)}(s) \ge \Gamma^A_d\right)=0.
\end{equation}
This proves the second convergence statement in the lemma.
For the first statement,
define the events
$
\mathcal{A}^*(d) := \{Y_i(0) \ge 7A \sqrt{i}(\log i) \text{ for all } i \ge d\}, \ d \in \mathbb{N}.
$
By assumption \eqref{farpicond},
\begin{equation}\label{f2.5}
\mathbb{P}\left(\mathcal{A}^*(d)\right) \rightarrow 1 \text{ as } d \rightarrow \infty.
\end{equation}
For any $d \in \mathbb{N}$, on the set $\mathcal{A}^*(d)$, for all $i \ge d$,
$$
\frac{Y_{i}(0) - \Gamma^A_d}{\sqrt{t^A_d}} \ge \frac{7A \sqrt{i}(\log i) - 6A \sqrt{d}(\log d)}{\sqrt{A d (\log d)}} \ge \sqrt{A \log i}.
$$
Hence, using \eqref{inf}, for all $d \ge 2$, on the event $\mathcal{A}^*(d)$,
\begin{multline}\label{f3}
\mathbb{P}\left( \inf_{s \in [0, t^A_d]} \inf_{i \ge d}Y_{i}(s) \le \Gamma^A_d \ \vert \ \Y(0)\right) \le 2\sum_{i \ge d}\bar{\Phi}\left(\sqrt{A \log i}\right)\\
\le \sum_{i \ge d}\frac{2}{\sqrt{2\pi A\log i}}e^{-A(\log i)/2} \le \sum_{i \ge d} i^{-A/2} \le  \frac{1}{d-1}.
\end{multline}
Hence, from \eqref{f2.5} and \eqref{f3}, 
\begin{equation}\label{f4}
\limsup_{d \rightarrow \infty}\mathbb{P}\left(\inf_{s \in [0, t^A_d]} \inf_{i \ge d}Y_{i}(s) \le \Gamma^A_d\right) =0.
\end{equation}
This proves the first convergence statement in the lemma and completes the  proof of  the lemma.
\end{proof}

The next few lemmas will be used to study the case when the starting configuration satisfies \eqref{stara}.

\begin{lemma}\label{farpiainf}
Fix $\theta >0$. Suppose there exists $i_{\theta} \ge 2,$ such that 
$Y_i(0) \ge \theta \log i$ for all $i \ge i_{\theta}$. Then there exists $\delta_{\theta} \in (0,1)$ and positive constants $C_1,C_2$ (depending on $\theta$) such that, a.s.,
\begin{equation*}
\mathbb{P}\left(\inf_{s \in [0, \delta_{\theta} t]} \inf_{i \ge e^t} Y_i(s) \le \theta t/2 \ \vert \ \Y(0)\right) \le C_1 e^{-C_2 t}, \ \text{ for all } t \ge 0.
\end{equation*} 
\end{lemma}

\begin{proof}
For any $i \ge e^t \vee i_{\theta}$, $\delta \in (0,1)$, by the assumption on the starting configuration,
\begin{align*}
\bar{\Phi}\left(\frac{Y_i(0) - \theta t/2}{\sqrt{\delta t}}\right) &\le \bar{\Phi}\left(\frac{\theta \log i - \theta t/2}{\sqrt{\delta t}}\right)\\
&\le \bar{\Phi}\left(\frac{\theta \log i}{2\sqrt{\delta t}}\right) \le \frac{1}{\sqrt{2 \pi}} \frac{2\sqrt{\delta t}}{\theta \log i} \exp\left\lbrace - \frac{(\log i)^2 \theta ^2}{8 \delta t}\right\rbrace.
\end{align*}
Hence, using \eqref{inf} with $d = \lceil e^t\rceil$, $\Gamma = \theta t/2$ and $t=\delta t$, we obtain $t_0>0$ such that for all $t \ge t_0$,
\begin{align*}
\mathbb{P}\left( \inf_{s \in [0, \delta t]} \inf_{i \ge e^t}Y_{i}(s) \le \theta t/2 \ \vert \ \Y(0)\right) &\le 2\sum_{i \ge e^t}\bar{\Phi}\left(\frac{Y_{i}(0) - \theta t/2}{\sqrt{\delta t}}\right) \le 2\sum_{i \ge e^t}\exp\left\lbrace - \frac{(\log i)^2 \theta ^2}{8 \delta t}\right\rbrace\\
&=2 \sum_{j=1}^{\infty}\sum_{e^{jt} \le i < e^{(j+1)t}}\exp\left\lbrace - \frac{(\log i)^2 \theta ^2}{8 \delta t}\right\rbrace\\
&\le 2 \sum_{j=1}^{\infty}e^{(j+1)t}\exp\left\lbrace - \frac{\theta ^2j^2 t}{8 \delta}\right\rbrace.
\end{align*}
By the above bound, we can obtain $\delta_{\theta} \in (0,1)$ and $C_1, C_2>0$ such that for all $t \ge 0$,
$$
\mathbb{P}\left( \inf_{s \in [0, \delta_{\theta} t]} \inf_{i \ge e^t}Y_{i}(s) \le \theta t/2 \ \vert \ \Y(0)\right) \le C_1 e^{-C_2 t},
$$
proving the lemma.
\end{proof}

\begin{lemma}\label{farpiasup}
For any $k \in \mathbb{N}, \alpha >0$,
$$
\lim_{t \rightarrow \infty} \mathbb{P}\left(\sup_{s \in [0,t]} Y_{(k)}(s) \ge \alpha t\right) = 0.
$$
\end{lemma}

\begin{proof}
For any $A \ge 1$, $t >0$, using \eqref{sup} with $\Gamma = \alpha t$, we obtain for any $l \ge k$,
\begin{align*}
\mathbb{P}\left( \sup_{s \in [0, t]} Y_{(k)}(s) \ge \alpha t, Y_{l}(0) \le A\right) \le 2\bar{\Phi}\left(\frac{l(\alpha t - A)/3 - t - Al}{\sqrt{lt}} \right) + 4(k+1)\bar{\Phi}\left(\frac{\alpha t - A}{3\sqrt{t}} \right).
\end{align*}
Hence, choosing and fixing any $l \ge k \vee (18/\alpha)$, for all $t \ge 16A/(3\alpha)$,
\begin{equation*}
\mathbb{P}\left( \sup_{s \in [0, t]} Y_{(k)}(s) \ge \alpha t, Y_{l}(0) \le A\right) \le 2\bar{\Phi}\left(\frac{\sqrt{t}}{2\sqrt{l}}\right) + 4(k+1) \bar{\Phi}\left(\frac{13 \alpha \sqrt{t}}{48}\right).
\end{equation*}
Hence,
$$
\limsup_{t \rightarrow \infty}\mathbb{P}\left( \sup_{s \in [0, t]} Y_{(k)}(s) \ge \alpha t\right) \le \mathbb{P}\left(Y_{l}(0) > A\right).
$$
As $A \ge 1$ is arbitrary, the lemma follows upon taking a limit as $A \rightarrow \infty$ in the above bound.
\end{proof}


Recall that, for $\vv = (v_1,v_2,\dots)^T \in \mathbb{R}_+^{\infty}$, $s(\vv)$ denotes the vector in $\mathbb{R}_+^{\infty}$ whose $i$-th coordinate is $v_1 + \dots + v_i$, $i \in \mathbb{N}$. \textcolor{black}{In the following, for a $\mathbb{R}_+^{\infty}$-valued random vector $\vv$, we will say $s(\vv)$ satisfies \eqref{eq:initinteg} if, almost surely, \eqref{eq:initinteg} holds with $x_0=0$ and $x_i = (s(\vv))_i$ for $i \ge 1$.}

\begin{lemma}\label{start}
Fix $a>0$. Let $\mu$ be a probability measure on $\RR_+^{\infty}$ such that 
 there exists a coupling $(\U, \V_a)$ of $\mu$ and $\pi_a$ such that, almost surely, \eqref{stara} holds. Then, almost surely, there exists $ i_0 \ge 2$ such that $s(\U)_i \ge (4a)^{-1} \log i$ for all $i \ge i_0$. In particular, $s(\U)$ satisfies \eqref{eq:initinteg}, and consequently $\mu \in \cls$. Furthermore, $s(\U \wedge \V_a)$ and $s(\V_a)$ satisfy \eqref{eq:initinteg} as well. 
\end{lemma}

\begin{proof}
We will show that, almost surely, there exists $i'_0 \ge 2$ such that
\begin{equation}\label{m1}
\sum_{j=1}^{i} V_{a,j} \ge \frac{1}{2a} \log i \ \text{ for all } i \ge i'_0.
\end{equation}
It will then follow, by \eqref{stara}, that there is a $i_0 \ge i'_0$ such that for all $i \ge i_0$,
\begin{equation}\label{eq:eq410}
s(\U)_i = \sum_{j=1}^{i} U_j \ge \sum_{j=1}^{i} V_{a,j} - \sum_{j=1}^{i} |V_{a,j} - U_j| \ge \frac{1}{2a} \log i - \frac{\log i}{\log \log i} \ge \frac{1}{4a} \log i.
\end{equation}
The same argument shows that $s(\U\wedge \V_a)_i \ge (4a)^{-1} \log i$ for $i\ge i_0$ as well.
Thus, in order to prove the lemma, it suffices to prove \eqref{m1}. Note that $E_j := (2+ ja) V_{a,j}$, $j \in \mathbb{N}$, are iid exponential random variables with mean one. Define
$$
S_i := \sum_{j=1}^i V_{a,j} = \sum_{j=1}^i \frac{E_j}{2+ ja}, \ i \in \mathbb{N}.
$$
Note that
$$
\mathbb{E}\left(S_i\right) = \sum_{j=1}^i \frac{1}{2+ ja} \ge \int_1^{i+1}\frac{1}{2+as}ds = \frac{1}{a}\log\left(\frac{2 + a(i+1)}{2 + a}\right).
$$
Moreover,
$$
\operatorname{Var}(S_i) = \sum_{j=1}^i \frac{1}{(2+ ja)^2} \le \sum_{j=1}^{\infty} \frac{1}{(2+ ja)^2} < \infty.
$$
Hence, $M_i := S_i - \mathbb{E}(S_i), i \in \mathbb{N}$, is an $L^2$-bounded martingale. Therefore, by the martingale convergence theorem \cite[Theorem 11.10]{klenke2013probability}, there exists a random variable $M_{\infty}$ with 
$\mathbb{E}(M_{\infty}^2) < \infty$ such that
$
M_i \to M_{\infty}$  as  $i \rightarrow \infty$, a.s. and in $L^2$.
Thus, almost surely, there exists $\underline{M} \in (-\infty, \infty)$ and $i''_0 \ge 2$ such that
$$
S_i \ge \frac{1}{a}\log\left(\frac{2 + a(i+1)}{2 + a}\right) - \underline{M}, \ \text{ for all } \ i \ge i''_0.
$$
The estimate in \eqref{m1}, and hence the lemma, is immediate from the above.
\end{proof}

The following lemma uses the above three lemmas to quantify the influence of far away coordinates when the starting configuration satisfies \eqref{stara}.
Note that Lemma \ref{start} shows that, if $\U$ is as in that lemma, then the distribution of $(0,s(\U)^T)^T$ is in $\cls_0$.
\begin{lemma}\label{farpia}
Fix $a>0$. 
Let $\mu$ be a probability measure on $\RR_+^{\infty}$ such that 
 there exists a coupling $(\U, \V_a)$ of $\mu$ and $\pi_a$ such that, almost surely, \eqref{stara} holds.
 Let $\YY(\cdot)$ be the infinite Atlas model given as the unique weak solution of
 \eqref{Atlasdef} with $\YY(0)$ distributed as  $(0,s(\U)^T)^T$. Then there exists $\delta_0 \in (0,1)$ such that,
\begin{align*}
\limsup_{d \rightarrow \infty}\mathbb{P}\left(\inf_{s \in [0, \delta_0 \log d]} \inf_{i \ge d}Y_{i}(s) \le \frac{1}{8a}\log d\right) = 0,
\end{align*}
and for any $k \in \mathbb{N}$
\begin{align*}
\limsup_{d \rightarrow \infty}\mathbb{P}\left(\sup_{s \in [0, \delta_0 \log d]} Y_{(k)}(s) \ge \frac{1}{8a} \log d\right) = 0.
\end{align*}
\end{lemma}

\begin{proof}
Take $\delta_0 := \delta_{(4a)^{-1}}$ as defined in Lemma \ref{farpiainf}. The second limit above follows by Lemma \ref{farpiasup}. To prove the first limit, define the event $\tilde{\mathcal{E}} := \{\exists \ i_0 \ge 2 \text{ such that } s(\U)_i \ge (4a)^{-1} \log i \text{ for all } i \ge i_0\}$.
By Lemma \ref{farpiainf}, with $\theta = (4a)^{-1}$,
$$
 \limsup_{d \rightarrow \infty}\mathbb{P}\left(\inf_{s \in [0, \delta_0 \log d]} \inf_{i \ge d}Y_{i}(s) \le \frac{1}{8a}\log d\right) \le \mathbb{P}\left(\tilde{\mathcal{E}}^c\right).
$$
The right hand side above is zero by Lemma \ref{start}. This proves the lemma.
\end{proof}

\section{Analyzing excursions of synchronous couplings}\label{excsec}

In this section, we will estimate the rate of decay of the $L^1$ distance between synchronously coupled infinite ordered Atlas processes by defining and analyzing suitable excursions of the processes. Each such excursion will capture an event which ensures that the $L^1$ distance decreases by a fixed amount. Theorems \ref{piconv} and \ref{piaconv} will then be proved by estimating the number of such excursions on the time interval $[0,t]$ for large $t$.

Let $(\U_1, \U_2)$ be an $\mathbb{R}_+^{\infty} \times \mathbb{R}_+^{\infty}$-valued random variable such that $\U_1 \le \U_2$ and $s(\U_1)$ (and thus $s(\U_2)$) satisfies \eqref{eq:initinteg}. Consider the synchronously coupled copies $\XX^{\U_1}(\cdot)$ and $\XX^{\U_2}(\cdot)$ of the infinite ordered Atlas model (in the sense of Proposition \ref{syncprop} (ii)) starting from $(0,s(\U_1)^T)^T$ and $(0,s(\U_2)^T)^T$ respectively, driven by Brownian motions $\{B_i\}_{i \in \NN_0}$ (that are independent of $\U_1$, $\U_2$).  Let $\Z^{\U_1}(\cdot)$ and $\Z^{\U_2}(\cdot)$ be the associated gap processes as introduced in Proposition \ref{syncprop} (iii), and let $\Delta \Z(\cdot) := \Z^{\U_2}(\cdot) - \Z^{\U_1}(\cdot)$.

Fix $k \in \mathbb{N}$. For $s \ge 0$, define the following stopping times: $T^k_0(s) := s$, and
\begin{align}\label{Tdef}
T^k_j(s) := \inf\{u \ge T^k_{j-1}(s): \Delta Z_{k-j+1}(u) =0\}, \ \ 1 \le j \le k.
\end{align}
Set $\mathcal{T}_k(s) := T^k_k(s)$. We will analyze excursions of $\Delta \Z$
the coupled processes defined by the following stopping times. Fix $\epsilon>0$. Define $\sigma_0 := 0$, $\sigma_1 := \inf\{s \ge 0: \Delta Z_k(s) \ge \epsilon\}$, and for $j \ge 0$,
\begin{align*}
\sigma_{2j+2} &:= \inf\{s \ge \mathcal{T}_k\left(\sigma_{2j+1}\right) : \Delta Z_k(s) = 0\},\\
\sigma_{2j+3} &:= \inf\{s \ge \sigma_{2j+2} : \Delta Z_k(s) \ge \epsilon\},
\end{align*}
where the infimum in \eqref{Tdef} and in the above two definitions is taken to be $\infty$ if the corresponding sets are empty. \textcolor{black}{Although the stopping times above depend on $k$, we suppress this dependence for notational convenience.}
For $T \ge 0$, define the number of excursions until time $T$ by
\begin{equation*}
\N_T :=
\left\{
	\begin{array}{ll}
		 \sup\{j \ge 0: \sigma_{2j+1} \le T\} + 1 & \mbox{if } \sigma_1 \leq T,\\
		0  & \mbox{if } \sigma_1 > T. 
	\end{array}
\right.
\end{equation*}
\textcolor{black}{In the next two lemmas, we work with a fixed $k \in \mathbb{N}$ and stopping times $\{\sigma_j : j \ge 0\}$ defined for this fixed $k$.} 

Each of the above excursions ensure a decrease of the $L^1$ norm of $\Delta \Z(\cdot)$ by a fixed amount as described by the following lemma.

\begin{lemma}\label{l1dec}
Assume $\sum_{j=1}^{\infty} j\Delta Z_j(0) < \infty$. Then, for any $j \ge 0$, when $\sigma_{2j+2} <\infty$,
\begin{equation*}
\sum_{i=1}^{\infty} \Delta Z_i(\sigma_{2j+2}) - \sum_{i=1}^{\infty} \Delta Z_i(\sigma_{2j+1}) \le -\epsilon/2^k.
\end{equation*}
\end{lemma}
\begin{proof}
By Lemma \ref{l1control},
\begin{equation*}
\sum_{i=1}^{\infty} \Delta Z_i(t) = \sum_{i=1}^{\infty} \Delta Z_i(0) + \frac{1}{2} \Delta L_1(t), \ t \ge 0.
\end{equation*}
Thus, it suffices to show that
\begin{equation}\label{ld1}
\Delta L_1(\sigma_{2j+2}) - \Delta L_1(\sigma_{2j+1}) \le -\epsilon/2^{k-1}.
\end{equation}
Recalling the stopping times $\{T^k_l(\cdot)\}_{0 \le l \le k}$ from \eqref{Tdef}, note that for any $1 \le i \le k$, $j \in \mathbb{N}_0$,
\begin{multline}\label{ld2}
0= \Delta Z_i\left(T^k_{k-i+1}\left(\sigma_{2j+1}\right)\right) = \Delta Z_i(\sigma_{2j+1}) + \left(\Delta L_i\left(T^k_{k-i+1}\left(\sigma_{2j+1}\right)\right) - \Delta L_i(\sigma_{2j+1})\right)\\
\qquad \qquad \qquad \qquad \qquad \qquad -\frac{1}{2} \left(\Delta L_{i-1}\left(T^k_{k-i+1}\left(\sigma_{2j+1}\right)\right) - \Delta L_{i-1}(\sigma_{2j+1})\right)\\
  -\frac{1}{2} \left(\Delta L_{i+1}\left(T^k_{k-i+1}\left(\sigma_{2j+1}\right)\right) - \Delta L_{i+1}(\sigma_{2j+1})\right).
\end{multline}
Recalling that $\U_1 \le \U_2$, by Proposition \ref{syncprop} (iii), for each $i \in \mathbb{N}$, $j \in \mathbb{N}_0$, $\Delta Z_i(\sigma_{2j+1}) \ge 0$ and $\Delta L_i(\cdot)$ is non-positive and non-increasing. Thus, from \eqref{ld2} we obtain for any $1 \le i \le k$,
\begin{align*}
-\left(\Delta L_i\left(T^k_{k-i+1}\left(\sigma_{2j+1}\right)\right) - \Delta L_i(\sigma_{2j+1})\right) &\ge -\frac{1}{2} \left(\Delta L_{i+1}\left(T^k_{k-i+1}\left(\sigma_{2j+1}\right)\right) - \Delta L_{i+1}(\sigma_{2j+1})\right)\\
&\ge -\frac{1}{2} \left(\Delta L_{i+1}\left(T^k_{k-i}\left(\sigma_{2j+1}\right)\right) - \Delta L_{i+1}(\sigma_{2j+1})\right).\notag
\end{align*}
Using the above bound recursively for $i=1,\dots, k-1$, we get for any $j \in \mathbb{N}_0$,
\begin{align}\label{ld3}
-\left(\Delta L_1\left(T^k_{k}\left(\sigma_{2j+1}\right)\right) - \Delta L_1(\sigma_{2j+1})\right) &\ge -\frac{1}{2} \left(\Delta L_{2}\left(T^k_{k-1}\left(\sigma_{2j+1}\right)\right) - \Delta L_{2}(\sigma_{2j+1})\right)\\
&\ge \cdots \ge -\frac{1}{2^{k-1}}\left(\Delta L_{k}\left(T^k_{1}\left(\sigma_{2j+1}\right)\right) - \Delta L_{k}(\sigma_{2j+1})\right).\notag
\end{align}
Moreover, from \eqref{ld2} with $i=k$, again using Proposition \ref{syncprop} (iii),
\begin{equation}\label{ld4}
-\left(\Delta L_{k}\left(T^k_{1}\left(\sigma_{2j+1}\right)\right) - \Delta L_{k}(\sigma_{2j+1})\right) \ge \Delta Z_k(\sigma_{2j+1}) \ge \epsilon.
\end{equation}
Recalling $\sigma_{2j+2} \ge \mathcal{T}_k\left(\sigma_{2j+1}\right) = T^k_k\left(\sigma_{2j+1}\right)$ and $\Delta L_1(\cdot)$ is non-increasing, \eqref{ld1} follows from \eqref{ld3} and \eqref{ld4}. This proves the lemma.
\end{proof}

The following lemma will be used to control the maximum length of excursions in a large time interval.

\begin{lemma}\label{exclength}
Suppose there exists $D \ge 1$ such that the distribution of $\U_2/D$ is stochastically dominated by $\pi$.
Let $T_k := \inf\{S \ge e: 48D^2 k(k+1)^2 \log T \le T \text{ for all } T \ge S\}$. Then, for any $n \in \mathbb{N}$, $T \ge T_k$,
\begin{multline*}
\mathbb{P}\left(\sigma_1 \le T, \ \sigma_{2j+2} - \sigma_{2j+1} > 48D^2k(k+1)^2 \log T \text{ for some } 0 \le j \le \N_T - 1\right)\\
\le \mathbb{P}\left(\N_T > n\right) + 5kn T^{- 2}.
\end{multline*}
\end{lemma}

\begin{proof}
We can assume without loss of generality that, in addition to $\U_1$, $\U_2$, we are given another $\RR_+^{\infty}$-valued random variable $\V$ (on the same probability space) such that
$\U_1 \le \U_2 \le D\V$. In addition to infinite ordered Atlas models $\XX^{\U_i}$, with $\XX^{\U_i}(0) = \U_i$ for $i=1,2$, driven by Brownian motions $\{B_i\}$, we also construct infinite ordered Atlas models $\XX^{D\V}$ and $\XX^{[\gamma]}$ with drifts $1$  and $\gamma = 1/D$ respectively (the latter defined as in Proposition \ref{syncprop} (iv)), satisfying $\XX^{D\V}(0) = \XX^{[\gamma]}(0) = D\V$, and 
driven by the same Brownian motions $\{B_i\}$ (that are independent of $\U_1, \U_2, \V$). The associated gap processes are denoted as $\Z^{D\V}$ and $\Z^{[\gamma]}$. In order to note the dependence of $\gamma$ on $D$, we will denote the  processes $\XX^{[\gamma]}$, $\Z^{[\gamma]}$ as $\XX^{\{D\}}$,
$\Z^{\{D\}}$ respectively.

By Proposition \ref{syncprop} (iii) and (iv),
\begin{equation}\label{e1}
\Z^{\U_2}(t) \le \Z^{D\V}(t) \le \Z^{\{D\}}(t) \ \text{ for all } t \ge 0.
\end{equation}
For $u \ge 0$, define
$$
S_i(u) := \inf\{s \ge u: Z^{\{D\}}_i(s) = 0\}  \text{ for } 1 \le i \le k, \ \ S(u) := \max_{1 \le i \le k}S_i(u).
$$
Write $A:= 24D^2 k(k+1)$. Note that $2A(k+1) \log T \le T$ for all $T \ge T_k$. Suppose for some $j \ge 0$, $T \ge T_k$, the following event holds:
$$
\mathcal{E}_{T,j} := \{\sigma_{2j+1} \le T, \sigma_{2j+2} - \sigma_{2j+1} > 2A(k+1) \log T\}.
$$ 
Then, either
\begin{equation}\label{e2}
\text{there is } 1 \le l \le k \text{ such that } T^k_l(\sigma_{2j+1}) - T^k_{l-1}\left(\sigma_{2j+1}\right) > 2A \log T,
\end{equation}
or
\begin{equation}\label{e3}
\sigma_{2j+2} - \mathcal{T}_k\left(\sigma_{2j+1}\right) = T^k_1\left(\mathcal{T}_k\left(\sigma_{2j+1}\right)\right) - \mathcal{T}_k\left(\sigma_{2j+1}\right)  > 2A \log T.
\end{equation}
Suppose \eqref{e2} holds. Let $l$ be the minimum integer in $\{1,\dots,k\}$ for which \eqref{e2} holds. Then, recalling $\sigma_{2j+1} \le T$,
$$
T^k_{l-1}\left(\sigma_{2j+1}\right) \le \sigma_{2j+1} + 2A(l-1) \log T \le T + 2A(k-1) \log T \le 2T - A\log T
$$
where the last inequality is true because $T \ge T_k$. Similarly, if \eqref{e3} holds and \eqref{e2} does not hold for any $1 \le l \le k$, then
$$
\mathcal{T}_k\left(\sigma_{2j+1}\right) \le \sigma_{2j+1} + 2Ak \log T \le 2T - A \log T.
$$
Observe that for any $1 \le l \le k$, $j \ge 0$, by \eqref{e1},
\begin{align*}
T^k_l(\sigma_{2j+1}) &:= \inf\{u \ge T^k_{l-1}\left(\sigma_{2j+1}\right) : \Delta Z_{k-l+1}(u) = 0\}\\
& \le \inf\{u \ge T^k_{l-1}\left(\sigma_{2j+1}\right) : Z^{\{D\}}_{k-l+1}(u) = 0\} = S_{k-l+1}\left(T^k_{l-1}\left(\sigma_{2j+1}\right)\right),
\end{align*}
and similarly,
$$
\sigma_{2j+2} = T^k_1\left(\mathcal{T}_k\left(\sigma_{2j+1}\right)\right) \le S_k\left(\mathcal{T}_k\left(\sigma_{2j+1}\right)\right).
$$
Hence, for any $T \ge T_k$,
\begin{equation*}
\mathcal{E}_{T,j} \subseteq \mathcal{E}'_{T} := \{\exists \ t \in [0, 2T - A \log T] \text{ such that } S(t) - t > 2A \log T\}.
\end{equation*}
From this observation and the union bound, for any $n \in \mathbb{N}$, $T \ge T_k$,
\begin{align}\label{e4}
&\mathbb{P}\left(\sigma_1 \le T, \ \sigma_{2j+2} - \sigma_{2j+1} > 2A(k+1) \log T \text{ for some } 0 \le j \le \N_T - 1\right)\\
&\le \mathbb{P}\left(\N_T > n\right) + \sum_{j=0}^{n-1}\mathbb{P}\left(\mathcal{E}_{T,j}\right)
\le \mathbb{P}\left(\N_T > n\right) + n \mathbb{P}\left(\mathcal{E}'_{T}\right).\notag
\end{align}
In the rest of the proof, we will estimate $\mathbb{P}\left(\mathcal{E}'_{T}\right)$. 

For any $T>0$, define $n(T) := \lceil2T/(A \log T) \rceil \vee 2$ and
$$
t_l := A l \log T, \ 0 \le l \le n(T).
$$
Then $\cup_{l=0}^{n(T)} [t_{l}, t_{l+1}] \supset [0,2T]$ and each interval $[t_{l}, t_{l+1}]$ is
of length $A \log T$.

 Suppose $S(t) - t > 2A \log T$ for some $t \in [t_l, t_{l+1}]$, $0 \le l \le n(T) - 2$. Then there exists $1 \le i \le k$ such that $S_i(t_{l+1}) - t_{l+1} > A \log T$. Hence, for any $T \ge T_k$,
\begin{equation}\label{e5}
\mathbb{P}\left(\mathcal{E}'_{T}\right) \le \mathbb{P}\left(\exists \ 0 \le l \le n(T) - 2, 1 \le i \le k \text{ such that } S_i(t_{l+1}) - t_{l+1} > A \log T\right).
\end{equation}
Note that, by scaling properties of the Atlas model, the process $\Z^{\{D\}}(\cdot)$ is stationary with
$
\Z^{\{D\}}(t) \sim \otimes_{i=1}^{\infty}\operatorname{Exp}(2/D)
$
for all $t \ge 0$. Define for any $T \ge e$ the event
$$
\mathcal{A}_T := \{Z^{\{D\}}_i(t_l) \le 24 D\log T \text{ for all } 1 \le i \le k, \ 0 \le l \le n(T) - 1\}.
$$
By the stationarity of $\Z^{\{D\}}(\cdot)$ and the union bound, for any $T \ge e$,
\begin{equation}\label{e6}
\mathbb{P}\left(\mathcal{A}_T^c\right) \le k n(T) e^{-48 \log T} \le k\left(2 + \frac{2T}{A\log T}\right)T^{-48} \le 4kT^{-47}.
\end{equation}
Denote by 
$\mathbf{L}^{\{D\}} = (L^{\{D\}}_0(\cdot), L^{\{D\}}_1(\cdot), \dots)$ 
the collision local times associated with the infinite ordered Atlas model  $\XX^{\{D\}}$ (with drift $\gamma = 1/D$).
Then, for any $0 \le l \le n(T) - 1$, $1 \le i \le k$, $t \ge 0$, $T \ge e$, 
\begin{equation}\label{e7}
\{S_i(t_l) - t_l >t\} = \{X^{\{D\}}_{i-1}(s) < X^{\{D\}}_i(s) \text{ for all } s \in [t_l, t_l + t]\}.
\end{equation}
If the above event holds, then \textcolor{black}{$L_i(t_l + s) - L_i(t_l) = 0$} for all $0 \le s \le t$. Hence, using the fact that $X^{\{D\}}_j(\cdot)$'s are ranked and applying \eqref{order_SDEgamma}, we obtain for any $0 \le s \le t$,
on the event in \eqref{e7},
\begin{align*}
X^{\{D\}}_{i-1}(s+ t_l) \ge \frac{1}{i}\sum_{q=0}^{i-1} X^{\{D\}}_q(s + t_l) = \frac{s}{D i} + \frac{1}{i}\sum_{q=0}^{i-1} X^{\{D\}}_q(t_l) + \frac{1}{i} \sum_{q=0}^{i-1}\left(B_q(s + t_l) - B_q(t_l)\right),
\end{align*}
and
$$
X^{\{D\}}_{i}(s+ t_l) \le X^{\{D\}}_{i}(t_l) + \left(B_i(s + t_l) - B_i(t_l)\right).
$$
Moreover, if $\mathcal{A}_T$ holds, then for any $0 \le l \le n(T) - 1$, $1 \le q \le i-1$,
$$
X^{\{D\}}_i(t_l) - X^{\{D\}}_q(t_l) = \sum_{p=q+1}^{i} Z^{\{D\}}_p(t_l) \le 24 D(i-q) \log T.
$$
Hence, defining the new Brownian motions
$$
\tilde{B}_1(s) := \frac{1}{\sqrt{i}} \sum_{q=0}^{i-1}\left(B_q(s + t_l) - B_q(t_l)\right), \ \tilde{B}_2(s) := \left(B_i(s + t_l) - B_i(t_l)\right), \ 0 \le s \le t,
$$
and using the above observations in \eqref{e7}, we obtain for any $0 \le l \le n(T) - 1$, $1 \le i \le k$, $t \ge 0$, $T \ge e$,
\begin{multline*}
\mathbb{P}\left(S_i(t_l) - t_l >t, \ \mathcal{A}_T\right) \le \mathbb{P}\left(\frac{t}{D i} + \frac{1}{i}\sum_{q=0}^{i-1} X^{\{D\}}_q(t_l) + \frac{1}{\sqrt{i}}\tilde{B}_1(t) <  X^{\{D\}}_{i}(t_l) + \tilde{B}_2(t), \ \mathcal{A}_T\right)\\
= \mathbb{P}\left(\frac{1}{\sqrt{i}}\tilde{B}_1(t) - \tilde{B}_2(t) < - \frac{t}{D i} + \frac{1}{i}\sum_{q=0}^{i-1} \left(X^{\{D\}}_i(t_l) - X^{\{D\}}_q(t_l)\right), \ \mathcal{A}_T\right)\\
\le \mathbb{P}\left(\frac{1}{\sqrt{i}}\tilde{B}_1(t) - \tilde{B}_2(t) < - \frac{t}{D i} + \frac{1}{i}\sum_{q=0}^{i-1} 24 D(i-q) \log T\right)\\
= \mathbb{P}\left(\frac{1}{\sqrt{i}}\tilde{B}_1(t) - \tilde{B}_2(t) < - \frac{t}{D i} + 12D(i+1) \log T\right).
\end{multline*}
Hence, taking $t=24D^2 i (i+1) \log T$ and using the fact that $\frac{1}{\sqrt{i}}\tilde{B}_1(24D^2 i (i+1) \log T) - \tilde{B}_2(24D^2 i (i+1) \log T)$ is a mean zero normal random variable with variance $24D^2(i+1)^2 \log T$, we obtain for any $0 \le l \le n(T) - 1$, $1 \le i \le k$, $T \ge e$,
\begin{multline}\label{e8}
\mathbb{P}\left(S_i(t_l) - t_l > 24D^2 i (i+1) \log T, \ \mathcal{A}_T\right) \le \bar\Phi\left(-\frac{\sqrt{24 \log T}}{2}\right)\\
\le \frac{2}{\sqrt{48\pi \log T}}\exp\left\lbrace - 24\log T/8\right\rbrace \le \frac{1}{2\sqrt{3\pi}} T^{- 3}.
\end{multline}
Thus, recalling that $A = 24D^2 k(k+1)$, using \eqref{e6} and \eqref{e8} in \eqref{e5}, we obtain for any $T \ge T_k$,
\begin{align*}
\mathbb{P}\left(\mathcal{E}'_{T}\right) &\le \mathbb{P}\left(\exists \ 0 \le l \le n(T) - 2, 1 \le i \le k \text{ such that } S_i(t_{l+1}) - t_{l+1} > A \log T\right)\\
&\le \sum_{l=0}^{n(T) - 2} \sum_{i=1}^k  \mathbb{P}\left(S_i(t_{l+1}) - t_{l+1} > 24D^2 i(i+1) \log T, \ \mathcal{A}_T\right) + \mathbb{P}\left(\mathcal{A}_T^c\right)\\
&\le n(T) k \frac{1}{2\sqrt{3\pi}} T^{- 3} +  4kT^{- 47} \le \frac{2k}{\sqrt{3\pi}}T^{-2} + 4kT^{- 47}.
\end{align*}
Finally, using the above bound in \eqref{e4}, for any $n \in \mathbb{N}, T \ge T_k$,
\begin{multline*}
\mathbb{P}\left(\sigma_1 \le T, \ \sigma_{2j+2} - \sigma_{2j+1} > 2A(k+1) \log T \text{ for some } 0 \le j \le \N_T - 1\right)\\
\le \mathbb{P}\left(\N_T > n\right) + \frac{2kn}{\sqrt{3\pi}}T^{-2} + 4kn T^{- 47} \le \mathbb{P}\left(\N_T > n\right) + 5kn T^{- 2},
\end{multline*}
which proves the lemma.
\end{proof}

\section{Proofs of main results}\label{mainproofs}

In this section we will prove Theorem \ref{piconv} and Theorem \ref{piaconv}, and their corollaries (i.e Corollary \ref{starcheck} and Corollary \ref{staracheck}). 
\subsection {Proof of Theorem \ref{piconv} }
The key idea in the proof of Theorem \ref{piconv} is to sandwich the initial gap vectors $\U$ 
and $\V$ between the $\mathbb{R}_+^{\infty}$-valued random vectors $\U \wedge \V$ and $\U \vee \V$. Using monotonicity properties of synchronous couplings described in Proposition \ref{syncprop}, one can bound the gap process $\Z^{\U}(\cdot)$ from above and below in a pathwise fashion by the synchronously coupled processes $\Z^{\U \vee \V}(\cdot)$ and $\Z^{\U \wedge \V}(\cdot)$ respectively. Lemma \ref{lesslem} uses the excursion estimates in Section \ref{excsec} to analyze synchronously coupled gap processes respectively started from $\V$ and a perturbation of $\V$ with the first $d$ entries replaced by those of $\U \wedge \V$. It shows that, for fixed $k \in \mathbb{N}, \epsilon>0$, the average time spent $\epsilon$ distance apart by the $k$-th coordinates of these two processes on a suitably large $d$-dependent time interval converges to zero in probability as $d \rightarrow \infty$. 

Lemma \ref{undconv} uses Lemma \ref{lesslem} and the quantitative estimates for the influence of far away coordinates obtained in Section \ref{farsec} to show that the averaged occupancy measure for the gap process started from $\U \wedge \V$ weakly converges to $\pi$. \textcolor{black}{The convergence of the averaged occupancy measure for the gap process started from $\U \vee \V$ is a consequence of the following lemma, which is a direct corollary of \cite[Theorem 4.7]{AS}. Theorem \ref{piconv} follows from these observations.} 
\textcolor{black}{
\begin{lemma}[Corollary of Theorem 4.7 of \cite{AS}]\label{AS4.7}
Let $(\U_{\circ}, \V)$ be an $\mathbb{R}_+^{\infty} \times \mathbb{R}_+^{\infty}$-valued random variable such that $\U_{\circ} \ge \V$ and $\V \sim \pi$. Consider the synchronously coupled copies $\XX^{\U_{\circ}}(\cdot)$ and $\XX^{\V}(\cdot)$ of the infinite ordered Atlas model (in the sense of Proposition \ref{syncprop} (ii)) starting from $(0,s(\U_{\circ})^T)^T$ and $(0,s(\V)^T)^T$ respectively, and let $\Z^{\U_{\circ}}(\cdot)$ and $\Z^{\V}(\cdot)$ denote the associated gap processes as in Proposition \ref{syncprop} (iii). Then the law of $\Z^{\U_{\circ}}(t)$ converges weakly to $\pi$ as $t \rightarrow \infty$. 
\end{lemma}}

\begin{lemma}\label{lesslem}
Suppose that $\U$ and $\V$ are as in the statement of Theorem \ref{piconv} .
%
For $d \in \mathbb{N}$, define the vector $\U_{(d)}$ by
$$
U_{(d), i} := (U_i \wedge V_i)\mathbf{1}(i \le d) + V_i\mathbf{1}(i>d), \ i \in \mathbb{N}.
$$ 
Then, $s(\U \wedge \V), s(\U), s(\V), s(\U_{(d)})$ all satisfy \eqref{eq:initinteg}.
Define infinite ordered Atlas models $\XX^{\U}, \XX^{\V}$ and $\XX^{\U_{(d)}}$,
starting from $(0,s(\U)^T)^T$, $(0,s(\V)^T)^T$ and  $(0,s(\U_{(d)})^T)^T$ respectively,
 as in Proposition \ref{syncprop} (i) using an infinite sequence of Brownian motions $\{B_i\}_{i\in \NN_0}$ (independent of $\U, \V$). Let $\Z^{\U}, \Z^{\V}, \Z^{\U_{(d)}}$ be the associated gap processes.
%
Then, for any $k \in \mathbb{N}$, $\epsilon \in (0,1)$, there exists $A_0\ge 1$ such that for any $\vartheta>0$, $1 \le j \le k$,
 \begin{equation*}
\mathbb{P}\left(\frac{1}{A_0 d \log d}\int_0^{A_0 \vartheta d \log d}\mathbf{1}\left(|Z^{\V}_j(s) - Z^{\U_{(d)}}_j(s)| > \epsilon \right)ds \ge 1\right) \rightarrow 0 \ \text{ as } \ d \rightarrow \infty.
 \end{equation*}
\end{lemma}

\begin{proof}
First note that, by  \eqref{star}, almost surely, there exists $i' \in \mathbb{N}$ such that $s(\U \wedge \V)_i \ge \sqrt{i}\log i$ for all $i \ge i'$. In particular, $s(\U \wedge \V)$, and hence $s(\U)$, $s(\V)$, satisfies \eqref{eq:initinteg}. As $s(\U_{(d)}) \ge s(\U \wedge \V)$ and $s(\U \wedge \V)$ satisfies \eqref{eq:initinteg}, we have that $s(\U_{(d)})$ also satisfies \eqref{eq:initinteg}.

Fix any $k \in \mathbb{N}, \epsilon \in (0,1)$, $\vartheta>0$. Write $\Delta \Z'(\cdot) := \Z^{\V}(\cdot) - \Z^{\U_{(d)}}(\cdot)$.

Recall the stopping times $\{\sigma_j : j \ge 0\}$ and the random variables $\{\N_T : T \ge 0\}$ from Section \ref{excsec} with our choice of $\epsilon$, $\U_1 = \U_{(d)}$ and $\U_2 = \V$. For $A \ge 4$, $d \in \mathbb{N}$, recall from Lemma \ref{farpi}  that
$t^A_d:= A d \log d$, and write $\N_{A,\vartheta,d} := \N_{\vartheta t^A_d}$. 

As $\sum_{j=1}^{\infty}j\Delta Z'_j(0) \le \sum_{j=1}^{d} j V_j < \infty$, by Lemma \ref{l1dec}, for any $j \ge 0$, when $\sigma_{2j+2}<\infty$,
\begin{equation*}
\sum_{i=1}^{\infty} \Delta Z'_i(\sigma_{2j+2}) - \sum_{i=1}^{\infty} \Delta Z'_i(\sigma_{2j+1}) \le -\epsilon/2^k.
\end{equation*}
Recalling the monotonicity property in Proposition \ref{syncprop}(iii), for any $A > 0$,
\begin{align}\label{main1}
\mathbb{P}\left(\N_{A,\vartheta,d} > 2^k \epsilon^{-1} d + 1 \right) \le \mathbb{P}\left(\sum_{i=1}^d\Delta Z_i'(0) > d \right) \le \mathbb{P}\left(\sum_{i=1}^dV_i > d\right) \le e^{-d/8}.
\end{align}
For $d \in \mathbb{N}$, $A \ge 4$, define the event
$$
\mathcal{E}_{d,A} := \{\sigma_1 \le \vartheta t^A_d, \ \sigma_{2j+2} - \sigma_{2j+1} \le 48 k(k+1)^2 \log  (\vartheta t^A_d) \text{ for all } 0 \le j \le \N_{A,\vartheta,d} - 1\}.
$$
By Lemma \ref{exclength} (with $D=1$, $T= \vartheta t^A_d$ and $n=2^k \epsilon^{-1} d + 1$) and \eqref{main1}, there exists $d_0 \in \mathbb{N}$ such that for any $d \ge d_0$, $A \ge 4$,
\begin{equation}\label{main2}
\mathbb{P}\left(\mathcal{E}_{d,A}^c \cap \{\sigma_1 \le \vartheta t^A_d\}\right) \le e^{-d/8} + \frac{10k2^k\epsilon^{-1}}{A^2 \vartheta^2 d(\log d)^2}.
\end{equation}
For $d \ge A \vee \vartheta \ge 4$, on the event $\mathcal{E}_{d,A}$,
\begin{align}\label{main3}
\frac{1}{t^A_d}\int_0^{\vartheta t^A_d} \mathbf{1}\left(|\Delta Z'_k(s)| \ge \epsilon\right) ds &\le \frac{1}{t^A_d}\sum_{j=0}^{\N_{A,\vartheta,d}-1}\int_{\sigma_{2j+1}}^{\sigma_{2j+2}}\mathbf{1}\left(|\Delta Z'_k(s)| \ge \epsilon\right)ds\\
&\le 48 k(k+1)^2 \left(\frac{\log (\vartheta t^A_d)}{t^A_d}\right) \N_{A,\vartheta,d} \le \frac{192 k(k+1)^2}{A d}\N_{A,\vartheta,d}.\notag
\end{align}
Fix $A = A_0(k) := 385 \epsilon^{-1} 2^{k} k(k+1)^2$ (suppressing dependence on $\epsilon$ for notational convenience). Note that this does not depend on $\vartheta$. Writing $t_{d,k} := t^{A_0(k)}_d$ and $\mathcal{E}_{d,k}:= \mathcal{E}_{d,A_0(k)}$, and using \eqref{main1} and \eqref{main3}, we obtain for $d \ge A_0(k) \vee \vartheta$,
\begin{align}\label{main4}
\mathbb{P}\left(\frac{1}{t_{d,k}}\int_0^{\vartheta t_{d,k}} \mathbf{1}\left(|\Delta Z'_k(s)| \ge \epsilon\right)ds \ge 1, \ \mathcal{E}_{d,k}\right)
&\le \mathbb{P}\left(\frac{192 k(k+1)^2}{A_0(k) d}\N_{A_0(k),\vartheta,d} \ge 1\right)\\
&\le \mathbb{P}\left(\N_{A_0(k),\vartheta,d} > 2^k \epsilon^{-1} d + 1 \right) \le e^{-d/8}.\notag
\end{align}
Thus, from \eqref{main2} and \eqref{main4}, for any $\vartheta>0$, and 
$d$ sufficiently large,
\begin{equation}\label{main5}
\mathbb{P}\left(\frac{1}{t_{d,k}}\int_0^{\vartheta t_{d,k}} \mathbf{1}\left(|\Delta Z'_k(s)| \ge \epsilon\right)ds \ge 1\right) \le 2e^{-d/8} + \frac{10k2^k\epsilon^{-1}}{A_0(k)^2 \vartheta^2 d(\log d)^2}  \rightarrow 0 \ \text{ as } \ d \rightarrow \infty.
\end{equation}
Note that the above holds for any $1 \le j \le k$, upon replacing $k$ with $j$. Thus, for any $1 \le j \le k$, writing $\vartheta_k := \vartheta t_{d,k}/t_{d,1} = \vartheta A_0(k)/A_0(1) = \vartheta 2^{k-3}k(k+1)^2$,
$$
\mathbb{P}\left(\frac{1}{t_{d,k}}\int_0^{\vartheta t_{d,k}} \mathbf{1}\left(|\Delta Z'_j(s)| \ge \epsilon\right)ds \ge 1\right) \le \mathbb{P}\left(\frac{1}{t_{d,j}}\int_0^{\vartheta_k t_{d,j}} \mathbf{1}\left(|\Delta Z'_j(s)| \ge \epsilon\right)ds \ge 1\right),
$$
which converges to zero as $d \rightarrow \infty$. 
This proves the lemma with $A_0 := A_0(k)$.
\end{proof}

\begin{lemma}\label{undconv}
Suppose that $\U$ and $\V$ are as in Theorem \ref{piconv}. Define the infinite ordered Atlas model $\XX^{\U\wedge \V}$
with $\XX^{\U\wedge \V}(0) = (0,s(\U\wedge \V)^T)^T$ and the same sequence of Brownian motions used to define
$\XX^{\U}, \XX^{\V}$ and $\XX^{\U_{(d)}}$ in Lemma \ref{lesslem}. Let $\Z^{\U\wedge \V}$ be the corresponding gap process.
%
Define the measure
 \begin{equation}\label{und}
\underline{\mu}_t(F) := \frac{1}{t}\int_0^t \mathbb{P}\left(\Z^{\U \wedge \V}(s) \in F\right)ds, \ F \in \mathcal{B}\left(\mathbb{R}_+^{\infty}\right), t >0.
\end{equation} 
 Then $\underline{\mu}_t$ converges weakly to $\pi$ as $t \rightarrow \infty$.
\end{lemma}

\begin{proof}
Fix any $\epsilon \in (0,1), \delta \in (0,1), k \in \mathbb{N}$. Recall $A_0$ from Lemma \ref{lesslem} for this choice of $\epsilon,k$. For $t > \delta^{-1} A_0 e$, write $d(t) := \lfloor \delta A_0^{-1} t/\log (\delta A_0^{-1} t)\rfloor$. Choose $t_0 > \delta^{-1} A_0 e$ large enough such that
$
d(t) \ge \delta A_0^{-1} t/[2\log (\delta A_0^{-1} t)] > k
$
and
$
\log (\delta A_0^{-1} t) \ge 2\log (2\log (\delta A_0^{-1} t))
$
for all $t \ge t_0$. Also, set $\vartheta := 4\delta^{-1}$. Note that for all $t \ge t_0$,
$$
A_0 d(t) \log d(t) \le \delta t, \ \ A_0 \vartheta d(t) \log d(t) \ge t.
$$
Recall the gap processes $\Z^{\V}, \Z^{\U_{(d)}}$ from Lemma \ref{lesslem}.
Then, for all $ t \ge t_0$, $1 \le j \le k$,
 \begin{multline*}
 \mathbb{P}\left(\frac{1}{t}\int_0^{t}\mathbf{1}\left(|Z^{\V}_j(s) - Z^{\U_{(d(t))}}_j(s)| > \epsilon \right)ds \ge \delta\right)\\
\le \mathbb{P}\left(\frac{1}{A_0 d(t) \log d(t)}\int_0^{A_0 \vartheta d(t) \log d(t)}\mathbf{1}\left(|Z^{\V}_j(s) - Z^{\U_{(d(t))}}_j(s)| > \epsilon \right)ds \ge 1\right).
 \end{multline*}
 Hence, applying Lemma \ref{lesslem} with $d(t)$ in place of $d$ and the above choice of $\vartheta$, we obtain for any $1 \le j \le k$,
 \begin{equation}\label{down1}
 \mathbb{P}\left(\frac{1}{t}\int_0^{t}\mathbf{1}\left(|Z^{\V}_j(s) - Z^{\U_{(d(t))}}_j(s)| > \epsilon \right)ds \ge \delta\right) \rightarrow 0 \ \text{ as } \ t \rightarrow \infty.
 \end{equation}
 Define the measures $\{\tilde\mu^{(k)}_t, t >0\}$ on $\mathbb{R}_+^k$ by
\begin{equation*}
\tilde\mu^{(k)}_t(F) := \frac{1}{t}\int_0^t \mathbb{P}\left(\Z^{\U_{(d(t))}}(s) \vert_k \in F\right)ds, \ F \in \mathcal{B}\left(\mathbb{R}_+^k\right), t >0,
\end{equation*} 
Take any closed set $F \in \mathcal{B}\left(\mathbb{R}_+^k\right)$ and, for $\eta>0$, let $F^{\eta} := \{y \in \mathbb{R}_+^k : \|y-x\|_1 \le \eta \text{ for some } x \in F\}$ be the closed $\eta$-neighborhood of $F$ with respect to $L^1$ distance. Observe that
\begin{align*}
\tilde\mu_t^{(k)}(F) 
 \le \frac{1}{t}\int_0^t \mathbb{P}\left(\Z^{\V}(s)\vert_k \in F^{k\epsilon}\right)ds  + \sum_{j=1}^k\frac{1}{t}\int_0^t\mathbb{P}\left(|Z^{\V}_j(s) - Z^{\U_{(d(t))}}_j(s)| > \epsilon \right)ds.
\end{align*}
Using the inequality
$$
\frac{1}{t}\int_0^t\mathbb{P}\left(|Z^{\V}_j(s) - Z^{\U_{(d(t))}}_j(s)| > \epsilon \right)ds  \le \delta + 
\mathbb{P}\left(\frac{1}{t}\int_0^{t}\mathbf{1}\left(|Z^{\V}_j(s) - Z^{\U_{(d(t))}}_j(s)| > \epsilon \right)ds \ge \delta\right)
$$
for $1 \le j \le k$, together with \eqref{down1}, and the stationarity of $\Z^{\V}(\cdot)$, we now obtain for any closed set $F \in \mathcal{B}\left(\mathbb{R}_+^k\right)$,
\begin{equation}\label{down2}
\limsup_{t \rightarrow \infty} \tilde\mu_t^{(k)}(F)  \le \pi^{(k)}(F^{k\epsilon}) + k\delta,
\end{equation}
where $ \pi^{(k)} = \bigotimes_{i=1}^{k} \operatorname{Exp}(2)$.

Our next step is to show the above bound \eqref{down2} with $\tilde\mu_t^{(k)}$ replaced by the measure $\underline{\mu}^{(k)}_t$ (the $k$-marginal of $\underline{\mu}_t$) by showing that, for any closed set $F \in \mathcal{B}\left(\mathbb{R}_+^k\right)$,
\begin{equation}\label{down3}
\underline{\mu}^{(k)}_t(F) - \tilde\mu_t^{(k)}(F)  \rightarrow 0 \ \text{ as } \ t \rightarrow \infty.
\end{equation}
To see this, take any $t \ge t_0$. Let $\gamma(t)$ be the probability law of $(0,s(\U_{(d(t))})^T)^T$ and denote by $\Y^{\gamma(t)}$ the unique weak solution of \eqref{rankbd} with $\Y^{\gamma(t)}(0)$ distributed as $\gamma(t)$, given on some probability space with driving Brownian motions $\{W_i\}_{i \in \NN_0}$. Denote the corresponding process of gaps (as defined in \eqref{gapdef}) as $\hat \Z^{\gamma(t)}$. 

Recall $d(t) > k$. For any $\zeta \in [0,t]$,
\begin{align}\label{eq:1125}
&\mathbb{P}\left(\Z^{\U_{(d(t))}}(\zeta) \vert_k \in F\right) = \mathbb{P}\left(\hat{\Z}^{\gamma(t)}(\zeta) \vert_k \in F\right)\\
 &= \mathbb{P}\left(\hat{\Z}^{\gamma(t)}(\zeta) \vert_k \in F, \ \inf_{s \in [0, t]} \inf_{i \ge d(t)}Y^{\gamma(t)}_{i}(s) > \sup_{s \in [0, t]} Y^{\gamma(t)}_{(k)}(s)\right)\notag\nonumber\\
 &\qquad + \mathbb{P}\left(\hat{\Z}^{\gamma(t)}(\zeta) \vert_k \in F, \inf_{s \in [0, t]} \inf_{i \ge d(t)}Y^{\gamma(t)}_{i}(s) \le \sup_{s \in [0, t]} Y^{\gamma(t)}_{(k)}(s)\right).\nonumber
 \end{align}
Recall from the Introduction that a finite-dimensional Atlas model has a unique strong solution. We denote by $\hat \Y^{\gamma(t), d(t)}$ such a process with $d(t)+1$ particles, driven by Brownian motions $\{W_i\}_{i=0}^{d(t)}$ and initial conditions
$\hat Y^{\gamma(t), d(t)}_i(0) = Y^{\gamma(t)}_i(0)$, $i=0, 1, \ldots, d(t)$. Denote by $\hat \XX^{\gamma(t), d(t)}$ the corresponding process of ordered particles
and denote the associated gap process by $\hat \Z^{\gamma(t), d(t)}$. Then, on the set, $\{\inf_{s \in [0, t]} \inf_{i \ge d(t)}Y^{\gamma(t)}_{i}(s) > \sup_{s \in [0, t]} Y^{\gamma(t)}_{(k)}(s)\}$,
 $\hat \Z^{\gamma(t)}(s)\vert_k = \hat \Z^{\gamma(t), d(t)}(s)\vert_k$ for all $s \in [0,t]$ and so, using \eqref{eq:1125}, for $\zeta \in [0,t]$,
 \begin{align*}
 	\mathbb{P}\left(\Z^{\U_{(d(t))}}(\zeta) \vert_k \in F\right) &=
	\mathbb{P}\left(\hat{\Z}^{\gamma(t), d(t)}(\zeta) \vert_k \in F, \ \inf_{s \in [0, t]} \inf_{i \ge d(t)}Y^{\gamma(t)}_{i}(s) > \sup_{s \in [0, t]} Y^{\gamma(t)}_{(k)}(s)\right)\\
	 &\qquad + \mathbb{P}\left(\hat{\Z}^{\gamma(t)}(\zeta) \vert_k \in F, \inf_{s \in [0, t]} \inf_{i \ge d(t)}Y^{\gamma(t)}_{i}(s) \le \sup_{s \in [0, t]} Y^{\gamma(t)}_{(k)}(s)\right)\\
	 &= \mathbb{P}\left(\hat{\Z}^{\gamma(t), d(t)}(\zeta) \vert_k \in F\right)\\
	 &\qquad- \mathbb{P}\left(\hat{\Z}^{\gamma(t), d(t)}(\zeta) \vert_k \in F, \ \inf_{s \in [0, t]} \inf_{i \ge d(t)}Y^{\gamma(t)}_{i}(s) \le \sup_{s \in [0, t]} Y^{\gamma(t)}_{(k)}(s)\right)\\
	 &\qquad + \mathbb{P}\left(\hat{\Z}^{\gamma(t)}(\zeta) \vert_k \in F, \inf_{s \in [0, t]} \inf_{i \ge d(t)}Y^{\gamma(t)}_{i}(s) \le \sup_{s \in [0, t]} Y^{\gamma(t)}_{(k)}(s)\right)
\end{align*}
which shows that
\begin{equation}\label{eq:42}
\left|\mathbb{P}\left(\Z^{\U_{(d(t))}}(\zeta) \vert_k \in F\right) - 	\mathbb{P}\left(\hat{\Z}^{\gamma(t), d(t)}(\zeta) \vert_k \in F\right)\right| \le \mathbb{P}\left( \inf_{s \in [0, t]} \inf_{i \ge d(t)}Y^{\gamma(t)}_{i}(s) \le \sup_{s \in [0, t]} Y^{\gamma(t)}_{(k)}(s)\right).
\end{equation}
Now let $\gamma^*$ be the probability law of $(0,s(\U\wedge \V)^T)^T$ and denote by $\Y^{\gamma^*}$ the unique weak solution of \eqref{rankbd} with $\Y^{\gamma^*}(0)$ distributed as $\gamma^*$. Denote the corresponding process of gaps  as $\hat \Z^{\gamma^*}$. 
The process of ordered particles $\hat \XX^{\gamma^*, d(t)}$ with $d(t)+1$ particles and initial values $\Y^{\gamma^*}(0)\vert_{d(t)}$, and the associated
gap process $\hat \Z^{\gamma^*, d(t)}$, are defined in a similar manner as $\hat{\XX}^{\gamma(t), d(t)}$ and $\hat{\Z}^{\gamma(t), d(t)}$. By a similar argument as above
\begin{equation}\label{eq:43}
\left|\mathbb{P}\left(\Z^{\U \wedge \V}(\zeta) \vert_k \in F\right) - 	\mathbb{P}\left(\hat{\Z}^{\gamma^*, d(t)}(\zeta) \vert_k \in F\right)\right| \le \mathbb{P}\left( \inf_{s \in [0, t]} \inf_{i \ge d(t)}Y^{\gamma^*}_{i}(s) \le \sup_{s \in [0, t]} Y^{\gamma^*}_{(k)}(s)\right).	
\end{equation}
Since $\U\wedge \V \vert_{d(t)} = \U_{(d(t))}\vert_{d(t)}$ we have that $\hat{\Z}^{\gamma^*, d(t)}$ and
$\hat{\Z}^{\gamma(t), d(t)}$ have the same distribution and so from \eqref{eq:42} and \eqref{eq:43} we have that
\begin{align}\label{eq:56}
	&\left|\mathbb{P}\left(\Z^{\U_{(d(t))}}(\zeta) \vert_k \in F\right)- \mathbb{P}\left(\Z^{\U \wedge \V}(\zeta) \vert_k \in F\right)\right|\\
	&\le \mathbb{P}\left( \inf_{s \in [0, t]} \inf_{i \ge d(t)}Y^{\gamma(t)}_{i}(s) \le \sup_{s \in [0, t]} Y^{\gamma(t)}_{(k)}(s)\right) + \mathbb{P}\left( \inf_{s \in [0, t]} \inf_{i \ge d(t)}Y^{\gamma^*}_{i}(s) \le \sup_{s \in [0, t]} Y^{\gamma^*}_{(k)}(s)\right).\nonumber
\end{align}

Now, recalling $t^{A_0\vartheta}_{d(t)}=A_0 \vartheta d(t) \log d(t) \ge t$ and that $\Gamma^{A_0\vartheta}_{d(t)} = 6A_0\vartheta \sqrt{d(t)}(\log d(t))$ from Lemma \ref{farpi}, we see that
\begin{align*}
&\mathbb{P}\left(\inf_{s \in [0, t]} \inf_{i \ge d(t)}Y^{\gamma(t)}_{i}(s) \le \sup_{s \in [0, t]} Y^{\gamma(t)}_{(k)}(s)\right)\\
&\le \mathbb{P}\left(\inf_{s \in [0, t^{A_0\vartheta}_{d(t)}]} \inf_{i \ge d(t)}Y^{\gamma(t)}_{i}(s) \le \sup_{s \in [0, t^{A_0\vartheta}_{d(t)}]} Y^{\gamma(t)}_{(k)}(s)\right)\\
 &\le \mathbb{P}\left(\inf_{s \in [0, t^{A_0\vartheta}_{d(t)}]} \inf_{i \ge d(t)}Y^{\gamma(t)}_{i}(s) \le \Gamma^{A_0\vartheta}_{d(t)}\right) + \mathbb{P}\left(\sup_{s \in [0, t^{A_0\vartheta}_{d(t)}]} Y^{\gamma(t)}_{(k)}(s) \ge \Gamma^{A_0\vartheta}_{d(t)}\right)\\
&\le \mathbb{P}\left(\inf_{s \in [0, t^{A_0\vartheta}_{d(t)}]} \inf_{i \ge d(t)}Y^{\gamma^*}_{i}(s) \le \Gamma^{A_0\vartheta}_{d(t)}\right) + \mathbb{P}\left(\sup_{s \in [0, t^{A_0\vartheta}_{d(t)}]} Y^{\pi}_{(k)}(s) \ge \Gamma^{A_0\vartheta}_{d(t)}\right)
\rightarrow 0 \ \text{ as } \ t \rightarrow \infty,
\end{align*}
where $\Y^{\pi}(\cdot)$ is the infinite Atlas model with $\Y^{\pi}(0)$ distributed as $(0, s(\V)^T)^T$.
The last inequality above uses the observation \textcolor{black}{$\U \wedge \V \le \U_{(d(t))} \le \V$, $\inf_{i \ge d(t)}Y^{\gamma^*}_{i}(s) = Y^{\gamma^*}_{(d(t))}(s)$} for all $s \ge 0$, and \cite[Corollary 3.10 (i)]{AS}. The limit follows from Lemma \ref{farpi} with $A=A_0\vartheta$, $d=d(t)$
upon using $\U \wedge \V \le \V$ and noting that \textcolor{black}{$\Y^{\gamma^*}(0)$ satisfies condition \eqref{farpicond} by virtue of assumption \eqref{star} and $\Y^{\pi}(0)$ satisfies \eqref{farpicond} by the law of large numbers.}

Similarly,
\begin{equation*}
\mathbb{P}\left(\inf_{s \in [0, t]} \inf_{i \ge d(t)}Y^{\gamma^*}_{i}(s) \le \sup_{s \in [0, t]} Y^{\gamma^*}_{(k)}(s)\right) \rightarrow 0 \ \text{ as } \ t \rightarrow \infty.
\end{equation*}
Using the above observations along with \eqref {eq:56}, we obtain
\begin{equation*}
\sup_{\zeta \in [0,t]}\left|\mathbb{P}\left(\Z^{\U \wedge \V}(\zeta) \vert_k \in F\right) - \mathbb{P}\left(\Z^{\U_{(d(t))}}(\zeta) \vert_k \in F\right)\right| \rightarrow 0 \ \text{ as } \ t \rightarrow \infty.
\end{equation*}
\eqref{down3} follows from this. From \eqref{down2} and \eqref{down3}, we obtain for any closed set $F \in \mathcal{B}\left(\mathbb{R}_+^k\right)$,
$$
\limsup_{t \rightarrow \infty} \underline{\mu}^{(k)}_t(F)  \le \pi^{(k)}(F^{k\epsilon}) + k\delta.
$$
As the left hand side above does not depend on $\epsilon,\delta$, we take a limit as $\epsilon,\delta$ go to zero in the above to obtain for any $k \in \mathbb{N}$,
$$
\limsup_{t \rightarrow \infty} \underline{\mu}^{(k)}_t(F)  \le \pi^{(k)}(F) \ \text{ for any  closed set} \  F \in \mathcal{B}\left(\mathbb{R}_+^k\right).
$$
The lemma now follows from the Portmanteau theorem \cite[Chapter 1, Theorem 2.1]{billingsley2013convergence}.
\end{proof}

Now, we are ready to prove Theorem \ref{piconv}.

\begin{proof}[Proof of Theorem \ref{piconv}]
	Let $(\U,\V)$ be as in the statement of Theorem \ref{piconv}. Let $\Z^{\U}, \Z^{\V}$ and $\Z^{\U\wedge \V}$ be as in Lemma \ref{lesslem} and Lemma \ref{undconv}.
By the monotonicity property in Proposition \ref{syncprop}(iii),
$$
\Z^{\U \wedge \V}(\cdot) \le \Z^{\U}(\cdot) \le \Z^{\U \vee \V}(\cdot).
$$
Define the measure
 \begin{equation*}
\overline{\mu}_t(F) := \frac{1}{t}\int_0^t \mathbb{P}\left(\Z^{\U \vee \V}(s) \in F\right)ds, \ F \in \mathcal{B}\left(\mathbb{R}_+^{\infty}\right), t >0,
\end{equation*} 
and recall $\underline{\mu}_t$ defined in \eqref{und}.
Also define
$$
{\mu}_t(F) := \frac{1}{t}\int_0^t \mathbb{P}\left(\Z^{\U }(s) \in F\right)ds, \ F \in \mathcal{B}\left(\mathbb{R}_+^{\infty}\right), t >0.
$$
Then, for any $k \in \mathbb{N}$, $\z \in \mathbb{R}_+^k$,
\begin{equation}\label{sand}
\overline{\mu}^{(k)}_t((-\infty,\z]) \le\mu^{(k)}_t((-\infty,\z]) 
\le \underline{\mu}^{(k)}_t((-\infty,\z]).
\end{equation}
By Lemma \ref{undconv}, $\underline{\mu}^{(k)}_t((-\infty,\z]) \rightarrow \pi^{(k)}((-\infty,\z])$ as $t \rightarrow \infty$. Moreover, as $\U \vee \V \ge \V$, \textcolor{black}{by Lemma \ref{AS4.7}}, the law of $\Z^{\U \vee \V}(t)$ converges weakly to $\pi$. In particular, $\overline{\mu}^{(k)}_t((-\infty,\z]) \rightarrow \pi^{(k)}((-\infty,\z])$ as $t \rightarrow \infty$. Hence, by \eqref{sand}, for any $k \in \mathbb{N}$, $\z \in \mathbb{R}_+^k$,
$$
\mu^{(k)}_t((-\infty,\z]) \rightarrow \pi^{(k)}((-\infty,\z]) \ \text{ as } \ t \rightarrow \infty.
$$
This proves the theorem.
\end{proof}

\begin{proof}[Proof of Corollary \ref{starcheck}]
The fact that (i) implies \eqref{star} is immediate on observing that we can find a coupling $(\U,\V)$ of $\mu$ and $\pi$ such that 
$\U\le \V$ a.s.

For the statement in  (ii), we will consider the independent coupling of $(\U,\V)$, that is, under the coupling, the marginals $\U$ and $\V$ are independent. Recall from Remark \ref{remthm1} (b) that the statement in \eqref{star2} is equivalent  to
\begin{equation}\label{star2n}
 \liminf_{d \rightarrow \infty} \frac{1}{\sqrt{d} (\log d)} \sum_{i=1}^d (U_i \wedge \frac{1}{2}) = \infty, \mbox{ a.s. }
 \end{equation}
 It suffices to show that, conditional on a realization of $\U= \mathbf{u}$ such that \eqref{star2n} holds with $\U = \mathbf{u}$, the independent random variables $\Psi_i := V_i \wedge u_i$ satisfy
\begin{equation}\label{c1}
\liminf_{m \rightarrow \infty}\frac{1}{\sqrt{m}\log m}\sum_{i=1}^m \Psi_i = \infty \ \text{ almost surely}.
\end{equation}
Note that
$$
\mu_i:=\mathbb{E}\left(\Psi_i\right) = \int_0^{u_i}2ve^{-2v}dv + \int_{u_i}^{\infty}2u_ie^{-2v}dv = \frac{1}{2}(1-e^{-2u_i}).
$$
As $x \wedge(1/2) \le 1-e^{-2x} \le 2\left(x \wedge(1/2)\right)$ for all $x \ge 0$, \eqref{star2n} with $\U = \mathbf{u}$ is equivalent to the condition
\begin{equation}\label{c2}
\lim_{m \rightarrow \infty}\frac{1}{\sqrt{m}\log m}\sum_{i=1}^m \mu_i = \infty. 
\end{equation}
Take any $C>0$. By \eqref{c2}, there exists $m_C \in \mathbb{N}$ such that $\sum_{i=1}^m \mu_i \ge 2C\sqrt{m}\log m$ for all $m \ge m_C$. Note that for any $k \in \mathbb{N}$,
$$
\mathbb{E}\left(|\Psi_i - \mu_i|^k\right) \le \textcolor{black}{2^k\mathbb{E}(V_i^k + \mu_i^k)} \le 2^k(k!2^{-k} + 2^{-k}) \le 2 \ k!.
$$
Therefore, $\Psi_i$ satisfies condition (2.16) in \cite[Chapter 2]{wainwright2019high} with $\sigma=2, b=1$. Hence, by \cite[Chapter 2, Proposition 2.3]{wainwright2019high}, $\mathbb{E}\left(e^{\lambda (\Psi_i - \mu_i)}\right) \le \exp\{2\lambda^2/(1- |\lambda|)\} \le \exp\{4\lambda^2\}$ for all $|\lambda| < 1/2$. 
Hence, using \cite[Chapter 2, Theorem 2.3]{wainwright2019high}, with $D_k = \Psi_k$, $b_k=2$ and $\nu_k = 2\sqrt{2}$, for all $m \ge m_C$ satisfying $C\sqrt{m}\log m \le 4m$,
\begin{align*}
\mathbb{P}\left(\sum_{i=1}^m \Psi_i \le C\sqrt{m}\log m\right) &
\le \mathbb{P}\left(\sum_{i=1}^m (\Psi_i - \mu_i) \le -C\sqrt{m}\log m\right)\\
&\le 2\exp\left\lbrace - \frac{(C \sqrt{m}\log m)^2}{16m}\right\rbrace = 2\exp\left\lbrace - C^2 (\log m)^2/16\right\rbrace.
\end{align*}
As the bound above is summable in $m$, we conclude from the Borel-Cantelli Lemma that for any $C>0$, almost surely, $\liminf_{m \rightarrow \infty}(\sqrt{m}\log m)^{-1}\sum_{i=1}^m \Psi_i \ge C$. This implies \eqref{c1} and thus proves  that the statement in (ii) implies \eqref{star}.

Consider now (iii). Thus $U_i = \lambda_i \Theta_i, \ i \in \mathbb{N}$, where $\{\Theta_i\}$ are iid non-negative random variables satisfying $\mathbb{P}(\Theta_1>0)>0$, and $\{\lambda_i\}$ are positive deterministic real numbers. 
If $\{\lambda_j\}$ satisfy condition (a), i.e. $\liminf_{j \rightarrow \infty} \lambda_j>0$,
there exists $\eta>0$ such that $U_i \wedge (1/2) \ge (\eta \Theta_i) \wedge (1/2)$ for all sufficiently large $i$. Note that $\{(\eta \Theta_i) \wedge (1/2)\}$ are iid with positive mean as $\mathbb{P}(\Theta_1>0)>0$. Therefore, \eqref{star2n} holds by the strong law of large numbers, implying \eqref{star}.

If $\{\lambda_j\}$ satisfy condition (b), namely $\limsup_{j \rightarrow \infty}\lambda_j < \infty$ and
\eqref{star3} holds, 
there exists $M \in [1,\infty)$ such that $\lambda_j \le M$ for all $j \in \mathbb{N}$. Note that $\mathbb{E}\left(U_i \wedge (1/2)\right) \ge \lambda_j\mathbb{E}\left(\Theta_i \wedge (1/2M)\right) \ge \lambda_j M^{-1} \mathbb{E}\left(\Theta_1 \wedge (1/2)\right)$. As $\mathbb{P}(\Theta_1>0)>0$, $\mathbb{E}\left(\Theta_1 \wedge (1/2)\right) >0$. Hence, by \eqref{star3}, for any $C>0$, there exists $m_C \in \mathbb{N}$ such that for all $m \ge m_C$,
$$
\sum_{i=1}^m\mathbb{E}\left(U_i \wedge (1/2)\right) \ge M^{-1} \mathbb{E}\left(\Theta_1 \wedge (1/2)\right)\sum_{i=1}^m \lambda_i \ge 2C \sqrt{m}\log m.
$$
As $|U_i \wedge (1/2) -\mathbb{E}\left(U_i \wedge (1/2)\right)| \le 1$ for all $i \in \mathbb{N}$,  by the Azuma-Hoeffding inequality \cite[Chapter 2, Corollary 2.1]{wainwright2019high}, for all $m \ge m_C$,
\begin{multline*}
\mathbb{P}\left(\sum_{i=1}^mU_i \wedge (1/2) \le C\sqrt{m}\log m\right)\\
\le \mathbb{P}\left(\sum_{i=1}^m (U_i \wedge (1/2) -\mathbb{E}\left(U_i \wedge (1/2)\right)) \le -C\sqrt{m}\log m\right) \le 2\exp\{-2C^2(\log m)^2\},
\end{multline*}
which is summable in $m$. Hence, we conclude from the Borel-Cantelli Lemma that $$\liminf_{m \rightarrow \infty}(\sqrt{m}\log m)^{-1}\sum_{i=1}^m U_i \wedge (1/2) = \infty$$ almost surely. Thus, \eqref{star2n} holds which shows \eqref{star}.

\end{proof}

\subsection{Proof of Theorem \ref{piaconv}}

We now prove Theorem \ref{piaconv}. As in the proof of Theorem \ref{piconv}, we will proceed by defining new starting configurations for the gap process in terms of $\U$ and $\V_a$ and exploiting the monotonicity properties of the synchronous coupling described in Proposition \ref{syncprop}.

\begin{proof}[Proof of Theorem \ref{piaconv}]
Let $(\U, \V_a)$ be as in the statement of Theorem \ref{piaconv}. 
%
First note that, by Lemma \ref{start}, $s(\U), s(\V_a)$ and $s(\U \wedge \V_a)$ all satisfy \eqref{eq:initinteg}. Define the vector $\U_{a,(d)}$ by
$$
U_{a,(d),i} := (U_i \wedge V_{a,i}) \mathbf{1}(i \le d) + (U_i \vee V_{a,i})\mathbf{1}(i>d), \ i \in \mathbb{N}, \ d \ge 2.
$$
As $\U_{a,(d)} \ge \U \wedge \V_a$ and  $s(\U \wedge \V_a)$ satisfies \eqref{eq:initinteg}, we have that $s(\U_{a,(d)})$ satisfies \eqref{eq:initinteg}. 

Observe that, with $V_i := (1+ ia/2)V_{a,i}, i \in \mathbb{N},$  $(\U_{a,(d)},\U \vee \V_a,\V)$ defines a coupling of $(\mu, \pi_a, \pi)$.
By the second condition in  \eqref{stara}, $\limsup_{d\to \infty} \frac{U_d}{V_d}<\infty$ a.s. and so the 
%
events
$$
\mathcal{E}^{(D)} := \{ \U \vee \V_a \le D \V\}, D \ge 1,
$$
satisfy
\begin{equation}\label{t21}
\mathbb{P}\left(\mathcal{E}^{(D)}\right) \rightarrow 1 \ \text{ as } \ D \rightarrow \infty.
\end{equation}
 Fix $D\ge 1$. Note that, for $i \in \NN$,
 $$s(\U_{a,(d)}\wedge D\V)_i = \sum_{j=1}^i (\U_{a,(d)})_j \wedge D\V_j \ge \sum_{j=1}^i V_{a,j} - \sum_{j=1}^i |V_{a,j}-U_j|.$$
 By a similar calculation as in \eqref{eq:eq410} it then follows that, for some $i'\in \NN$ and all $i\ge i'$,
 $s(\U_{a,(d)}\wedge D\V)_i \ge (4a)^{-1}\log i$. This shows that $s(\U_{a,(d)}\wedge D\V)$  satisfies \eqref{eq:initinteg}.
 Similarly,  $s((\U \vee \V_a)\wedge D\V)$ satisfies \eqref{eq:initinteg}.
 
 In the rest of the proof we will consider several random variables $\tilde \U$ such that $s(\tilde \U)$ satisfies \eqref{eq:initinteg} and are given on the same probability space.
 For all such  random variables we use Proposition \ref{syncprop} (ii) to construct infinite ordered Atlas models $\XX^{\tilde \U}(\cdot)$  starting
 from $(0,s(\tilde \U)^T)^T$   driven by a common collection of Brownian motions $\{B_i\}_{i \in \NN_0}$ (that are independent of all the 
 $\tilde\U$).  We denote by $\Z^{\tilde \U}(\cdot)$   the associated gap processes.

 Let $\Delta \Z^{*,D}(\cdot) := \Z^{(\U \vee \V_a) \wedge D\V}(\cdot) - \Z^{\U_{a,(d)} \wedge D\V}(\cdot)$.

Fix $\epsilon, \delta>0$, $k \in \mathbb{N}$. Recall the stopping times $\{\sigma_j : j \ge 0\}$ and the random variables $\{\N_T : T \ge 0\}$ from Section \ref{excsec} with $\U_1 = \U_{a,(d)} \wedge D\V$ and $\U_2 = (\U \vee \V_a) \wedge D\V$. Write $t'_d:= \log d$, and $\N'_{d} := \N_{t'_d}$. We suppress the dependence of $\{\sigma_j\}$ and $\N'_d$ on $D$ for notational convenience.

As $\sum_{j=1}^{\infty}j\Delta Z^{*,D}_j(0) \le \sum_{j=1}^{d} j |V_{a,j} - U_j| < \infty$, by Lemma \ref{l1dec}, for any $j \ge 0$, when $\sigma_{2j+2}<\infty$,
\begin{equation*}
\sum_{i=1}^{\infty} \Delta Z^{*,D}_i(\sigma_{2j+2}) - \sum_{i=1}^{\infty} \Delta Z^{*,D}_i(\sigma_{2j+1}) \le -\epsilon/2^k.
\end{equation*}
Hence, as in the proof of \eqref{main1}, for any $A>0$,
\begin{multline}\label{t22}
\mathbb{P}\left(\N'_{d} \ge 2^kA\frac{\log d}{\log \log d} \right) \le \mathbb{P}\left(\sum_{i=1}^d \Delta Z^{*,D}_i(0) \ge A \epsilon \frac{\log d}{\log \log d} - 2^{-k}\epsilon \right)\\
\le \mathbb{P}\left(\frac{\log \log d}{\log d}\sum_{i=1}^d |V_{a,i} - U_i| \ge A \epsilon - 2^{-k}\epsilon \frac{\log \log d}{\log d} \right) \rightarrow 0 \ \text{ as } \ d \rightarrow \infty
\end{multline}
where the claimed limit is a consequence of   \eqref{stara}.

For $d \ge 2$, define the events
$$
\mathcal{E}'_{d,D} := \{\sigma_1 \le \log d, \ \sigma_{2j+2} - \sigma_{2j+1} \le 48D^2 k(k+1)^2 \log \log d \text{ for all } 0 \le j \le \N'_{d} - 1\}.
$$
By Lemma \ref{exclength} with $T= \log d$, $n = \lfloor2^{k} \log d/\log\log d\rfloor$, for sufficiently large $d$,
\begin{equation*}
\mathbb{P}\left((\mathcal{E}'_{d,D})^c \cap \{\sigma_1 \le \log d\}\right) \le \mathbb{P}\left(\N'_{d} > 2^{k}\frac{\log d}{\log \log d} \right) + 5k\left(2^{k}\frac{\log d}{\log \log d}\right)(\log d)^{-2},
\end{equation*}
which converges to zero as $d \rightarrow \infty$ using \eqref{t22} with $A=1$.
On the event $\mathcal{E}'_{d,D}$,
\begin{align*}
\frac{1}{t'_d}\int_0^{t'_d} \mathbf{1}\left(|\Delta Z^{*,D}_k(s)| \ge \epsilon\right) ds &\le \frac{1}{t'_d}\sum_{j=0}^{\N'_{d}-1}\int_{\sigma_{2j+1}}^{\sigma_{2j+2}}\mathbf{1}\left(|\Delta Z^{*,D}_k(s)| \ge \epsilon\right)ds\\
&\le 48D^2 k(k+1)^2  \left(\frac{\log \log d}{\log d}\right) \N'_{d} .\notag
\end{align*}
Hence, using \eqref{t22} with $A= A_{D,k,\delta} := \delta [48D^2 \times 2^{k} k(k+1)^2]^{-1}$,
\begin{align}\label{t23}
&\mathbb{P}\left(\frac{1}{t'_d}\int_0^{t'_d} \mathbf{1}\left(|\Delta Z^{*,D}_k(s)| \ge \epsilon\right) ds > \delta\right)\\
& \le \mathbb{P}\left(\N'_{d} > 2^kA_{D,k,\delta}\frac{\log d}{\log \log d} \right) + \mathbb{P}\left((\mathcal{E}'_{d,D})^c \cap \{\sigma_1 \le \log d\}\right) \rightarrow 0 \ \text{ as } \ d \rightarrow \infty.\notag
\end{align}
Recalling the event $\mathcal{E}^{(D)}$ and using \eqref{t23},
\begin{multline*}
\limsup_{d \rightarrow \infty}\mathbb{P}\left(\frac{1}{t'_d}\int_0^{t'_d} \mathbf{1}\left(|Z^{\U \vee \V_a}_k(s) - Z^{\U_{a,(d)}}_k(s)| \ge \epsilon\right) ds > \delta\right)\\
\le \limsup_{d \rightarrow \infty}\mathbb{P}\left(\frac{1}{t'_d}\int_0^{t'_d} \mathbf{1}\left(|\Delta Z^{*,D}_k(s)| \ge \epsilon\right) ds > \delta\right) + \mathbb{P}\left(\left(\mathcal{E}^{(D)}\right)^c\right) = \mathbb{P}\left(\left(\mathcal{E}^{(D)}\right)^c\right).
\end{multline*}
As $D \ge 1$ is arbitrary, we obtain from \eqref{t21}, for any $\epsilon,\delta>0$, $k \in \mathbb{N}$,
\begin{equation}\label{t24}
\limsup_{d \rightarrow \infty}\mathbb{P}\left(\frac{1}{t'_d}\int_0^{t'_d} \mathbf{1}\left(|Z^{\U \vee \V_a}_k(s) - Z^{\U_{a,(d)}}_k(s)| \ge \epsilon\right) ds \ge \delta\right) = 0.
\end{equation}
Now, from the monotonicity property in Proposition \ref{syncprop}(iii),
\begin{align}\label{t25}
\Z^{\U \wedge \V_a}(\cdot) \le \Z^{\U}(\cdot) \le \Z^{\U \vee \V_a}(\cdot), \ \ \Z^{\U \wedge \V_a}(\cdot) \le \Z^{\V_a}(\cdot) \le \Z^{\U \vee \V_a}(\cdot).
\end{align}
Let $d'(t) := \lceil e^{t/\delta_0} \rceil$, $t >0,$ where $\delta_0$ is the constant appearing in Lemma \ref{farpia}. Define the following measures
\begin{align*}
\overline{\mu}_{a,t}(F) &:= \frac{1}{t}\int_0^t \mathbb{P}\left(\Z^{\U \vee \V_a}(s) \in F\right)ds, \ F \in \mathcal{B}\left(\mathbb{R}_+^{\infty}\right),\\
\underline{\mu}_{a,t}(F) &:= \frac{1}{t}\int_0^t \mathbb{P}\left(\Z^{\U \wedge \V_a}(s) \in F\right)ds, \ F \in \mathcal{B}\left(\mathbb{R}_+^{\infty}\right),
\end{align*}
for $t>0$.
For $k \in \mathbb{N}$, $t >0$, define the measure on $\mathbb{R}_+^k$,
\begin{align*}
\tilde\mu^{(k)}_{a,t}(F) := \frac{1}{t}\int_0^t \mathbb{P}\left(\Z^{\U_{a,(d'(t))}}(s) \vert_k \in F\right)ds, \ F \in \mathcal{B}\left(\mathbb{R}_+^k\right).
\end{align*}
Take any $k \in \mathbb{N}$ and $\z \in \mathbb{R}_+^k$. For $\eta \in \mathbb{R}$, write $\z^{\eta}$ for the vector obtained by adding $\eta$ to each coordinate. Note that, for any $\epsilon >0$,
 \begin{align}\label{t25.5}
\tilde\mu_t^{(k)}((-\infty,\z]) 
 \le \overline{\mu}^{(k)}_{a,t}((-\infty,\z^{\epsilon}])  + \sum_{j=1}^k\frac{1}{t}\int_0^t\mathbb{P}\left(|Z^{\U \vee \V_a}_j(s) - Z^{\U_{a,(d'(t))}}_j(s)| > \epsilon \right)ds.
\end{align}
Hence, using \eqref{t24},
\begin{equation*}
\limsup_{t \rightarrow \infty}\tilde\mu_t^{(k)}((-\infty,\z])  \le \limsup_{t \rightarrow \infty}\overline{\mu}^{(k)}_{a,t}((-\infty,\z^{\epsilon}]) \le \pi^{(k)}_a((-\infty,\z^{\epsilon}]),
\end{equation*}
where the last inequality uses $\Z^{\V_a}(\cdot) \le \Z^{\U \vee \V_a}(\cdot)$ (cf. Proposition \ref{syncprop} (iii)) and the stationarity of the process $\Z^{\V_a}(\cdot)$. Taking $\epsilon \downarrow 0$ above, we obtain for any $k \in \mathbb{N}, \z \in \mathbb{R}_+^k$,
\begin{equation}\label{t26}
\limsup_{t \rightarrow \infty}\tilde\mu_t^{(k)}((-\infty,\z]) \le \pi^{(k)}_a((-\infty,\z]).
\end{equation}
As in the proof of Theorem \ref{piconv}, we will show that for any $k \in \mathbb{N}, \z \in \mathbb{R}_+^k$,
\begin{equation}\label{t27}
\underline{\mu}^{(k)}_{a,t}((-\infty,\z]) - \tilde\mu_t^{(k)}((-\infty,\z]) \rightarrow 0 \ \text{ as } t \rightarrow \infty.
\end{equation}
Fix $t_0 >0$ such that $d'(t_0)>k$.  For $t\ge t_0$, let  $\gamma(t)$ be the probability law of $(0,s(\U_{a,(d'(t))})^T)^T$ and denote by $\Y^{\gamma(t)}$ the unique weak solution of \eqref{rankbd} with $\Y^{\gamma(t)}(0)$ distributed as $\gamma(t)$, given on some probability space with driving Brownian motions $\{W_i\}_{i \in \NN_0}$. Denote the corresponding process of gaps (as defined in \eqref{gapdef}) as $\hat \Z^{\gamma(t)}$. 
As in the proof of Lemma \ref{undconv} we use the fact that a finite-dimensional Atlas model has a unique strong solution. We denote by $\hat \Y^{\gamma(t), d'(t)}$ such a process with $d'(t)+1$ particles, driven by Brownian motions $\{W_i\}_{i=0}^{d'(t)}$ and initial conditions
$\hat Y^{\gamma(t), d'(t)}_i(0) = Y^{\gamma(t)}_i(0)$, $i=0, 1, \ldots, d'(t)$. Denote by 
$\hat \XX^{\gamma(t), d'(t)}$ the corresponding process of ordered particles
and denote the associated gap process by $\hat \Z^{\gamma(t), d'(t)}$. Exactly as in the proof of \eqref{eq:42} we now see that, for any $\zeta \in [0,t]$, $\z \in \mathbb{R}_+^k$,
\begin{multline}\label{t28}
\left|\mathbb{P}\left(\Z^{\U_{a,(d'(t))}}(\zeta) \vert_k \in (-\infty,\z]\right) - \mathbb{P}\left(\hat{\Z}^{\gamma(t), d'(t)}(\zeta) \vert_k \in (-\infty,\z]\right)\right|\\
\le \mathbb{P}\left(\inf_{s \in [0, t]} \inf_{i \ge d'(t)}Y^{\gamma(t)}_{i}(s) \le \sup_{s \in [0, t]} Y^{\gamma(t)}(s)\right).
\end{multline}
Now  let $\gamma^*$ be the probability law of $(0,s(\U\wedge \V_a)^T)^T$ and denote by $\Y^{\gamma^*}$ the unique weak solution of \eqref{rankbd} with $\Y^{\gamma^*}(0)$ distributed as $\gamma^*$. Denote the corresponding process of gaps  as $\hat \Z^{\gamma^*}$. 
The process of ordered particles $\hat \XX^{\gamma^*, d'(t)}$ with $d'(t)+1$ particles and initial values $\Y^{\gamma^*}(0)\vert_{d'(t)}$, and the associated
gap process $\hat \Z^{\gamma^*, d'(t)}$, are defined in a similar manner as $\hat{\XX}^{\gamma(t), d'(t)}$ and $\hat{\Z}^{\gamma(t), d'(t)}$. By a similar argument as above
\begin{multline}\label{t28.1}
\left|\mathbb{P}\left(\Z^{\U \wedge \V_a}(\zeta) \vert_k \in (-\infty,\z]\right) - \mathbb{P}\left(\hat{\Z}^{\gamma^*, d'(t)}(\zeta) \vert_k \in (-\infty,\z]\right)\right|\\
\le \mathbb{P}\left(\inf_{s \in [0, t]} \inf_{i \ge d'(t)}Y^{\gamma^*}_{i}(s) \le \sup_{s \in [0, t]} Y^{\gamma^*}_{(k)}(s)\right).
\end{multline}
Since $\U\wedge \V_a \vert_{d'(t)} = \U_{a,(d'(t))}\vert_{d'(t)}$ we have that $\hat{\Z}^{\gamma^*, d'(t)}$ and
$\hat{\Z}^{\gamma(t), d'(t)}$ have the same distribution and so from \eqref{t28} and \eqref{t28.1} we have that
\begin{equation}\label{eq:eq550}
\begin{aligned}
&\left|\mathbb{P}\left(\Z^{\U_{a,(d'(t))}}(\zeta) \vert_k \in (-\infty,\z]\right) - \mathbb{P}\left(\Z^{\U \wedge \V_a}(\zeta) \vert_k \in (-\infty,\z]\right)\right|\\
&\le \mathbb{P}\left(\inf_{s \in [0, t]} \inf_{i \ge d'(t)}Y^{\gamma(t)}_{i}(s) \le \sup_{s \in [0, t]} Y^{\gamma(t)}(s)\right)
+ \mathbb{P}\left(\inf_{s \in [0, t]} \inf_{i \ge d'(t)}Y^{\gamma^*}_{i}(s) \le \sup_{s \in [0, t]} Y^{\gamma^*}_{(k)}(s)\right).
\end{aligned}
\end{equation}

As in the proof of Theorem \ref{piconv},
\begin{multline*}
\mathbb{P}\left(\inf_{s \in [0, t]} \inf_{i \ge d'(t)}Y^{\gamma(t)}_{i}(s) \le \sup_{s \in [0, t]} Y^{\gamma(t)}_{(k)}(s)\right)\\
\le \mathbb{P}\left(\inf_{s \in [0, \delta_0 \log d'(t)]} \inf_{i \ge d'(t)}Y^{\gamma^*}_{i}(s) \le \frac{1}{8a}\log d'(t)\right) + \mathbb{P}\left(\sup_{s \in [0, \delta_0 \log d'(t)]} Y^{\tilde \gamma}_{(k)}(s) \ge \frac{1}{8a}\log d'(t)\right),
\end{multline*}
where  $\Y^{\tilde \gamma}(\cdot)$ is the infinite Atlas model with $\Y^{\tilde \gamma}(0)$ distributed as $(0, s(\U \vee \V_a)^T)^T$.
The last inequality above uses the observation $\U \wedge \V_a \le \U_{(d'(t)} \le \U \vee \V_a$ and \cite[Corollary 3.10 (i)]{AS}.
By Lemma \ref{farpia}, the right hand side above converges to zero as $d \rightarrow \infty$. Similarly,
$$
\mathbb{P}\left(\inf_{s \in [0, t]} \inf_{i \ge d'(t)}Y^{\gamma^*}_{i}(s) \le \sup_{s \in [0, t]} Y^{\gamma^*}_{(k)}(s)\right) \rightarrow 0 \ \text{ as} \ t \rightarrow \infty.
$$
Hence, from \eqref{eq:eq550}, for any $k \in \mathbb{N}$, $\z \in \mathbb{R}_+^k$,
$$
\sup_{\zeta \in [0,t]} \left|\mathbb{P}\left(\Z^{\U_{a,(d'(t))}}(\zeta) \vert_k \in (-\infty,\z]\right) - \mathbb{P}\left(\Z^{\U \wedge \V_a}(\zeta) \vert_k \in (-\infty,\z]\right)\right| \rightarrow 0 \ \text{ as } t \rightarrow \infty,
$$
which implies \eqref{t27}. By \eqref{t26} and \eqref{t27}, for any $k \in \mathbb{N}$, $\z \in \mathbb{R}_+^k$,
$$
\limsup_{t \rightarrow \infty}\underline{\mu}^{(k)}_{a,t}((-\infty,\z]) \le \pi^{(k)}_a((-\infty,\z]).
$$
Moreover, as $\Z^{\U \wedge \V_a}(\cdot) \le \Z^{\V_a}(\cdot)$, $\liminf_{t \rightarrow \infty}\underline{\mu}^{(k)}_{a,t}((-\infty,\z]) \ge \pi^{(k)}_a((-\infty,\z])$. Thus, we conclude that for any $k \in \mathbb{N}$, $\z \in \mathbb{R}_+^k$,
\begin{equation}\label{mm2}
\lim_{t \rightarrow \infty}\underline{\mu}^{(k)}_{a,t}((-\infty,\z]) = \pi^{(k)}_a((-\infty,\z]).
\end{equation}
Now, note that by \eqref{t25.5} and \eqref{t24}, for any $\epsilon>0, k \in \mathbb{N}$, $\z \in \mathbb{R}_+^k$
$$
\liminf_{t \rightarrow \infty}\tilde\mu_t^{(k)}((-\infty,\z^{-\epsilon}]) 
 \le \liminf_{t \rightarrow \infty}\overline{\mu}^{(k)}_{a,t}((-\infty,\z]).
$$
Also, by \eqref{t27}, $\liminf_{t \rightarrow \infty}\tilde\mu_t^{(k)}((-\infty,\z^{-\epsilon}])  = \liminf_{t \rightarrow \infty}\underline\mu^{(k)}_{a,t}((-\infty,\z^{-\epsilon}])$. Hence, using \eqref{mm2},
$$
\pi^{(k)}_a((-\infty,\z^{-\epsilon}]) = \liminf_{t \rightarrow \infty}\underline\mu^{(k)}_{a,t}((-\infty,\z^{-\epsilon}]) \le \liminf_{t \rightarrow \infty}\overline{\mu}^{(k)}_{a,t}((-\infty,\z]).
$$
As $\epsilon>0$ is arbitrary, we conclude $\pi^{(k)}_a((-\infty,\z]) \le \liminf_{t \rightarrow \infty}\overline{\mu}^{(k)}_{a,t}((-\infty,\z])$. Moreover, as $\U \vee \V_a \ge \V_a$, we have $\limsup_{t \rightarrow \infty}\overline{\mu}^{(k)}_{a,t}((-\infty,\z]) \le \pi^{(k)}_a((-\infty,\z])$. Hence, for any $k \in \mathbb{N}$, $\z \in \mathbb{R}_+^k$,
\begin{equation}\label{mm3}
\lim_{t \rightarrow \infty}\overline{\mu}^{(k)}_{a,t}((-\infty,\z]) = \pi^{(k)}_a((-\infty,\z]).
\end{equation}
Finally, as $\Z^{\U\wedge \V_a}(\cdot) \le \Z^{\U}(\cdot) \le \Z^{\U\vee \V_a}(\cdot)$, for any $k \in \mathbb{N}$, $\z \in \mathbb{R}_+^k$, $t >0$,
$$
\overline{\mu}^{(k)}_{a,t}((-\infty,\z]) \le \mu^{(k)}_t((-\infty,\z]) \le \underline{\mu}^{(k)}_{a,t}((-\infty,\z]),
$$
where 
$$
{\mu}_t(F) := \frac{1}{t}\int_0^t \mathbb{P}\left(\Z^{\U }(s) \in F\right)ds, \ F \in \mathcal{B}\left(\mathbb{R}_+^{\infty}\right), t >0.
$$
Thus, using \eqref{mm2} and \eqref{mm3}, we obtain for any $k \in \mathbb{N}$, $\z \in \mathbb{R}_+^k$,
$$
\lim_{t \rightarrow \infty}\mu^{(k)}_t((-\infty,\z]) = \pi^{(k)}_a((-\infty,\z]),
$$
which proves the theorem.
\end{proof}

\begin{proof}[Proof of Corollary \ref{staracheck}]
Suppose $\mu := \bigotimes_{i=1}^{\infty} \operatorname{Exp}(2 + ia + \lambda_i)$. Let $\V \sim \pi$. Consider the coupling $(\U,\V_a)$ of $\mu$ and $\pi_a$ defined as
$$
U_i := \frac{2V_i}{2 + ia + \lambda_i}, \ \ V_{a,i} := \frac{2V_i}{2 + ia}, \ \ i \in \mathbb{N}.
$$
For sufficiently large $i \in \mathbb{N}$,
\begin{equation}\label{ac0}
|V_{a,i} - U_i| = \frac{2|\lambda_i|}{(2+ia)(2 + ia + \lambda_i)} V_i \le \frac{2|\lambda_i|}{(1-\beta)(2+ia)^2} V_i  \le \frac{2}{a^2(1-\beta)}\frac{|\lambda_i|}{i^2} V_i.
\end{equation}
By \eqref{aexp},
\begin{equation}\label{ac1}
 \limsup_{d \rightarrow \infty} \frac{\log \log d}{\log d} \sum_{i=1}^d \mathbb{E}\left(\frac{|\lambda_i| V_i}{i^2}\right) = 0.
\end{equation}
Moreover, $M^*_i := \sum_{j=1}^i \left(\frac{|\lambda_j| V_j}{j^2} - \mathbb{E}\left(\frac{|\lambda_j| V_j}{j^2}\right)\right)$, $i \in \mathbb{N}$, is a martingale with uniformly bounded second moment.
Indeed, for $i \in \NN$,
$$
\mathbb{E}[(M^*_i)^2] = \sum_{j=1}^i\frac{|\lambda_j|^2}{4j^4}
\le \sum_{j=1}^{\infty}\frac{(C^2+1)(2+ja)^2}{4j^4}  <\infty.$$
Hence, by the martingale convergence theorem \cite[Theorem 11.10]{klenke2013probability}, $M^*_i$ almost surely converges to a finite limit as $i \rightarrow \infty$. This, along with \eqref{ac1}, implies that, almost surely,
$$
 \limsup_{d \rightarrow \infty} \frac{\log \log d}{\log d} \sum_{i=1}^d \frac{|\lambda_i| V_i}{i^2} = 0.
$$
By \eqref{ac0}, we conclude that the first statement in \eqref{stara} holds. 
Also note that
$$\frac{U_d}{dV_{a,d}} = \frac{2+da}{d(2+da+\lambda_d)} \le \frac{1}{(1-\beta)d}$$
which shows that the second statement in \eqref{stara} holds as well.
\end{proof}

 \section*{Acknowledgements}
The research of AB was supported in part by the NSF (DMS-1814894 and DMS-1853968). \textcolor{black}{We thank two anonymous referees whose valuable inputs significantly improved the article.}



\bibliographystyle{imsart-number} 
\bibliography{atlas_ref}

\begin{thebibliography}{32}

\bibitem{aldous41up}
\begin{barticle}[author]
\bauthor{\bsnm{Aldous},~\bfnm{David}\binits{D.}}
\btitle{``Up the River” game story. 2002}.
\bjournal{Download available from: http://www. stat. berkeley. edu/\~{}
  aldous/Research/OP/river. pdf}
\bvolume{41}
\bpages{42--47}.
\end{barticle}
\endbibitem

\bibitem{banerjee2020dimension}
\begin{barticle}[author]
\bauthor{\bsnm{Banerjee},~\bfnm{Sayan}\binits{S.}} \AND
  \bauthor{\bsnm{Brown},~\bfnm{Brendan}\binits{B.}}
(\byear{2020}).
\btitle{Dimension-free local convergence and perturbations for reflected
  {B}rownian motions}.
\bjournal{arXiv preprint arXiv:2009.12937}.
\end{barticle}
\endbibitem

\bibitem{BanBudh}
\begin{barticle}[author]
\bauthor{\bsnm{Banerjee},~\bfnm{Sayan}\binits{S.}} \AND
  \bauthor{\bsnm{Budhiraja},~\bfnm{Amarjit}\binits{A.}}
(\byear{2020}).
\btitle{Parameter and dimension dependence of convergence rates to stationarity
  for reflecting {B}rownian motions}.
\bjournal{Annals of Applied Probability}
\bvolume{30}
\bpages{2005--2029}.
\end{barticle}
\endbibitem

\bibitem{bass1987uniqueness}
\begin{barticle}[author]
\bauthor{\bsnm{Bass},~\bfnm{Richard~F}\binits{R.~F.}} \AND
  \bauthor{\bsnm{Pardoux},~\bfnm{Etienne}\binits{E.}}
(\byear{1987}).
\btitle{Uniqueness for diffusions with piecewise constant coefficients}.
\bjournal{Probability Theory and Related Fields}
\bvolume{76}
\bpages{557--572}.
\end{barticle}
\endbibitem

\bibitem{billingsley2013convergence}
\begin{bbook}[author]
\bauthor{\bsnm{Billingsley},~\bfnm{Patrick}\binits{P.}}
(\byear{2013}).
\btitle{Convergence of probability measures}.
\bpublisher{John Wiley \& Sons}.
\end{bbook}
\endbibitem

\bibitem{budlee}
\begin{barticle}[author]
\bauthor{\bsnm{Budhiraja},~\bfnm{Amarjit}\binits{A.}} \AND
  \bauthor{\bsnm{Lee},~\bfnm{Chihoon}\binits{C.}}
(\byear{2007}).
\btitle{Long time asymptotics for constrained diffusions in polyhedral
  domains}.
\bjournal{Stochastic processes and their applications}
\bvolume{117}
\bpages{1014--1036}.
\end{barticle}
\endbibitem

\bibitem{cabezas2019brownian}
\begin{barticle}[author]
\bauthor{\bsnm{Cabezas},~\bfnm{Manuel}\binits{M.}},
  \bauthor{\bsnm{Dembo},~\bfnm{Amir}\binits{A.}},
  \bauthor{\bsnm{Sarantsev},~\bfnm{Andrey}\binits{A.}} \AND
  \bauthor{\bsnm{Sidoravicius},~\bfnm{Vladas}\binits{V.}}
(\byear{2019}).
\btitle{Brownian Particles with Rank-Dependent Drifts: Out-of-Equilibrium
  Behavior}.
\bjournal{Communications on Pure and Applied Mathematics}
\bvolume{72}
\bpages{1424--1458}.
\end{barticle}
\endbibitem

\bibitem{DJO}
\begin{binproceedings}[author]
\bauthor{\bsnm{Dembo},~\bfnm{Amir}\binits{A.}},
  \bauthor{\bsnm{Jara},~\bfnm{Milton}\binits{M.}} \AND
  \bauthor{\bsnm{Olla},~\bfnm{Stefano}\binits{S.}}
(\byear{2019}).
\btitle{The infinite {A}tlas process: Convergence to equilibrium}.
In \bbooktitle{Annales de l'Institut Henri Poincar{\'e}, Probabilit{\'e}s et
  Statistiques}
\bvolume{55}
\bpages{607--619}.
\bpublisher{Institut Henri Poincar{\'e}}.
\end{binproceedings}
\endbibitem

\bibitem{dembo2016large}
\begin{barticle}[author]
\bauthor{\bsnm{Dembo},~\bfnm{Amir}\binits{A.}},
  \bauthor{\bsnm{Shkolnikov},~\bfnm{Mykhaylo}\binits{M.}},
  \bauthor{\bsnm{Varadhan},~\bfnm{SR~Srinivasa}\binits{S.~S.}} \AND
  \bauthor{\bsnm{Zeitouni},~\bfnm{Ofer}\binits{O.}}
(\byear{2016}).
\btitle{Large deviations for diffusions interacting through their ranks}.
\bjournal{Communications on Pure and Applied Mathematics}
\bvolume{69}
\bpages{1259--1313}.
\end{barticle}
\endbibitem

\bibitem{dembo2017equilibrium}
\begin{barticle}[author]
\bauthor{\bsnm{Dembo},~\bfnm{Amir}\binits{A.}} \AND
  \bauthor{\bsnm{Tsai},~\bfnm{Li-Cheng}\binits{L.-C.}}
(\byear{2017}).
\btitle{Equilibrium fluctuation of the Atlas model}.
\bjournal{The Annals of Probability}
\bvolume{45}
\bpages{4529--4560}.
\end{barticle}
\endbibitem

\bibitem{fernholz2002stochastic}
\begin{bincollection}[author]
\bauthor{\bsnm{Fernholz},~\bfnm{E~Robert}\binits{E.~R.}}
(\byear{2002}).
\btitle{Stochastic portfolio theory}.
In \bbooktitle{Stochastic portfolio theory}
\bpages{1--24}.
\bpublisher{Springer}.
\end{bincollection}
\endbibitem

\bibitem{fernholz2009stochastic}
\begin{barticle}[author]
\bauthor{\bsnm{Fernholz},~\bfnm{Robert}\binits{R.}} \AND
  \bauthor{\bsnm{Karatzas},~\bfnm{Ioannis}\binits{I.}}
(\byear{2009}).
\btitle{Stochastic portfolio theory: an overview}.
\bjournal{Handbook of numerical analysis}
\bvolume{15}
\bpages{89--167}.
\end{barticle}
\endbibitem

\bibitem{HR}
\begin{barticle}[author]
\bauthor{\bsnm{Harrison},~\bfnm{J~Michael}\binits{J.~M.}} \AND
  \bauthor{\bsnm{Reiman},~\bfnm{Martin~I}\binits{M.~I.}}
(\byear{1981}).
\btitle{Reflected {B}rownian motion on an orthant}.
\bjournal{The Annals of Probability}
\bpages{302--308}.
\end{barticle}
\endbibitem

\bibitem{harrison1987brownian}
\begin{barticle}[author]
\bauthor{\bsnm{Harrison},~\bfnm{J~Michael}\binits{J.~M.}} \AND
  \bauthor{\bsnm{Williams},~\bfnm{Ruth~J}\binits{R.~J.}}
(\byear{1987}).
\btitle{Brownian models of open queueing networks with homogeneous customer
  populations}.
\bjournal{Stochastics: An International Journal of Probability and Stochastic
  Processes}
\bvolume{22}
\bpages{77--115}.
\end{barticle}
\endbibitem

\bibitem{harrison1987multidimensional}
\begin{barticle}[author]
\bauthor{\bsnm{Harrison},~\bfnm{J~Michael}\binits{J.~M.}} \AND
  \bauthor{\bsnm{Williams},~\bfnm{Ruth~J}\binits{R.~J.}}
(\byear{1987}).
\btitle{Multidimensional reflected {B}rownian motions having exponential
  stationary distributions}.
\bjournal{The Annals of Probability}
\bpages{115--137}.
\end{barticle}
\endbibitem

\bibitem{ichiba2010collisions}
\begin{barticle}[author]
\bauthor{\bsnm{Ichiba},~\bfnm{Tomoyuki}\binits{T.}} \AND
  \bauthor{\bsnm{Karatzas},~\bfnm{Ioannis}\binits{I.}}
(\byear{2010}).
\btitle{On collisions of {B}rownian particles}.
\bjournal{Annals of Applied Probability}
\bvolume{20}
\bpages{951--977}.
\end{barticle}
\endbibitem

\bibitem{ichiba2013strong}
\begin{barticle}[author]
\bauthor{\bsnm{Ichiba},~\bfnm{Tomoyuki}\binits{T.}},
  \bauthor{\bsnm{Karatzas},~\bfnm{Ioannis}\binits{I.}} \AND
  \bauthor{\bsnm{Shkolnikov},~\bfnm{Mykhaylo}\binits{M.}}
(\byear{2013}).
\btitle{Strong solutions of stochastic equations with rank-based coefficients}.
\bjournal{Probability Theory and Related Fields}
\bvolume{156}
\bpages{229--248}.
\end{barticle}
\endbibitem

\bibitem{ichiba2017yet}
\begin{barticle}[author]
\bauthor{\bsnm{Ichiba},~\bfnm{Tomoyuki}\binits{T.}} \AND
  \bauthor{\bsnm{Sarantsev},~\bfnm{Andrey}\binits{A.}}
(\byear{2017}).
\btitle{Yet another condition for absence of collisions for competing
  {B}rownian particles}.
\bjournal{Electronic Communications in Probability}
\bvolume{22}.
\end{barticle}
\endbibitem

\bibitem{jourdain2008propagation}
\begin{barticle}[author]
\bauthor{\bsnm{Jourdain},~\bfnm{Benjamin}\binits{B.}} \AND
  \bauthor{\bsnm{Malrieu},~\bfnm{Florent}\binits{F.}}
(\byear{2008}).
\btitle{Propagation of chaos and {P}oincar{\'e} inequalities for a system of
  particles interacting through their {C}{D}{F}}.
\bjournal{Annals of Applied Probability}
\bvolume{18}
\bpages{1706--1736}.
\end{barticle}
\endbibitem

\bibitem{karatzas2016systems}
\begin{binproceedings}[author]
\bauthor{\bsnm{Karatzas},~\bfnm{Ioannis}\binits{I.}},
  \bauthor{\bsnm{Pal},~\bfnm{Soumik}\binits{S.}} \AND
  \bauthor{\bsnm{Shkolnikov},~\bfnm{Mykhaylo}\binits{M.}}
(\byear{2016}).
\btitle{Systems of {B}rownian particles with asymmetric collisions}.
In \bbooktitle{Annales de l'Institut Henri Poincar{\'e}, Probabilit{\'e}s et
  Statistiques}
\bvolume{52}
\bpages{323--354}.
\bpublisher{Institut Henri Poincar{\'e}}.
\end{binproceedings}
\endbibitem

\bibitem{KW}
\begin{barticle}[author]
\bauthor{\bsnm{Kella},~\bfnm{Offer}\binits{O.}} \AND
  \bauthor{\bsnm{Whitt},~\bfnm{Ward}\binits{W.}}
(\byear{1996}).
\btitle{Stability and structural properties of stochastic storage networks}.
\bjournal{Journal of Applied Probability}
\bpages{1169--1180}.
\end{barticle}
\endbibitem

\bibitem{klenke2013probability}
\begin{bbook}[author]
\bauthor{\bsnm{Klenke},~\bfnm{Achim}\binits{A.}}
(\byear{2013}).
\btitle{Probability Theory: A Comprehensive Course}.
\bpublisher{Springer Science \& Business Media}.
\end{bbook}
\endbibitem

\bibitem{lindvall1999strassen}
\begin{barticle}[author]
\bauthor{\bsnm{Lindvall},~\bfnm{Torgny}\binits{T.}}
(\byear{1999}).
\btitle{On {S}trassen's theorem on stochastic domination}.
\bjournal{Electronic communications in probability}
\bvolume{4}
\bpages{51--59}.
\end{barticle}
\endbibitem

\bibitem{mikosch1999regular}
\begin{bbook}[author]
\bauthor{\bsnm{Mikosch},~\bfnm{Thomas}\binits{T.}}
(\byear{1999}).
\btitle{Regular variation, subexponentiality and their applications in
  probability theory}
\bvolume{99}.
\bpublisher{Eindhoven University of Technology Eindhoven, The Netherlands}.
\end{bbook}
\endbibitem

\bibitem{PP}
\begin{barticle}[author]
\bauthor{\bsnm{Pal},~\bfnm{Soumik}\binits{S.}} \AND
  \bauthor{\bsnm{Pitman},~\bfnm{Jim}\binits{J.}}
(\byear{2008}).
\btitle{One-dimensional {B}rownian particle systems with rank-dependent
  drifts}.
\bjournal{The Annals of Applied Probability}
\bvolume{18}
\bpages{2179--2207}.
\end{barticle}
\endbibitem

\bibitem{sarantsev2015triple}
\begin{barticle}[author]
\bauthor{\bsnm{Sarantsev},~\bfnm{Andrey}\binits{A.}}
(\byear{2015}).
\btitle{Triple and simultaneous collisions of competing {B}rownian particles}.
\bjournal{Electronic journal of probability}
\bvolume{20}.
\end{barticle}
\endbibitem

\bibitem{AS}
\begin{binproceedings}[author]
\bauthor{\bsnm{Sarantsev},~\bfnm{Andrey}\binits{A.}}
(\byear{2017}).
\btitle{Infinite systems of competing {B}rownian particles}.
In \bbooktitle{Annales de l'Institut Henri Poincar{\'e}, Probabilit{\'e}s et
  Statistiques}
\bvolume{53}
\bpages{2279--2315}.
\bpublisher{Institut Henri Poincar{\'e}}.
\end{binproceedings}
\endbibitem

\bibitem{sarantsev2017stationary}
\begin{barticle}[author]
\bauthor{\bsnm{Sarantsev},~\bfnm{Andrey}\binits{A.}} \AND
  \bauthor{\bsnm{Tsai},~\bfnm{Li-Cheng}\binits{L.-C.}}
(\byear{2017}).
\btitle{Stationary gap distributions for infinite systems of competing
  {B}rownian particles}.
\bjournal{Electronic Journal of Probability}
\bvolume{22}.
\end{barticle}
\endbibitem

\bibitem{tang2018optimal}
\begin{barticle}[author]
\bauthor{\bsnm{Tang},~\bfnm{Wenpin}\binits{W.}} \AND
  \bauthor{\bsnm{Tsai},~\bfnm{Li-Cheng}\binits{L.-C.}}
(\byear{2018}).
\btitle{Optimal surviving strategy for drifted Brownian motions with
  absorption}.
\bjournal{The Annals of Probability}
\bvolume{46}
\bpages{1597--1650}.
\end{barticle}
\endbibitem

\bibitem{tsai_stat}
\begin{barticle}[author]
\bauthor{\bsnm{Tsai},~\bfnm{Li-Cheng}\binits{L.-C.}}
(\byear{2018}).
\btitle{{Stationary distributions of the Atlas model}}.
\bjournal{Electronic Communications in Probability}
\bvolume{23}
\bpages{1 -- 10}.
\bdoi{10.1214/18-ECP112}
\end{barticle}
\endbibitem

\bibitem{wainwright2019high}
\begin{bbook}[author]
\bauthor{\bsnm{Wainwright},~\bfnm{Martin~J}\binits{M.~J.}}
(\byear{2019}).
\btitle{High-dimensional Statistics: A Non-asymptotic Viewpoint}
\bvolume{48}.
\bpublisher{Cambridge University Press}.
\end{bbook}
\endbibitem

\bibitem{williams1995semimartingale}
\begin{barticle}[author]
\bauthor{\bsnm{Williams},~\bfnm{Ruth~J}\binits{R.~J.}}
(\byear{1995}).
\btitle{Semimartingale reflecting {B}rownian motions in the orthant}.
\bjournal{IMA Volumes in Mathematics and its Applications}
\bvolume{71}
\bpages{125--125}.
\end{barticle}
\endbibitem

\end{thebibliography}

\end{document}